\documentclass{amsart}
\usepackage{graphicx}
\usepackage{graphicx,cancel,xcolor,hyperref,comment,graphicx,geometry}

\usepackage[foot]{amsaddr}
\let\oldtocsection=\tocsection
 
\let\oldtocsubsection=\tocsubsection 
 
\let\oldtocsubsubsection=\tocsubsubsection
 
\renewcommand{\tocsection}[2]{\vspace{0.5em}\hspace{0em}\oldtocsection{#1}{#2}}
\renewcommand{\tocsubsection}[2]{\vspace{0.5em}\hspace{1em}\oldtocsubsection{#1}{#2}}
\renewcommand{\tocsubsubsection}[2]{\vspace{0.5em}\hspace{2em}\oldtocsubsubsection{#1}{#2}}
\usepackage{graphicx,cancel,xcolor,hyperref,comment,graphicx,geometry}

\usepackage{graphicx,url,etoolbox}

\usepackage{lipsum}

 \makeatletter
\patchcmd{\@settitle}{center}{flushleft}{}{}
\patchcmd{\@settitle}{center}{flushleft}{}{}
\patchcmd{\@setauthors}{\centering}{\raggedright}{}{}
\patchcmd{\abstract}{3pc}{0pt}{}{} 
\makeatother

\makeatletter
\renewcommand*\@maketitle{%
  \normalfont\normalsize
  \@adminfootnotes
  \@mkboth{\@nx\shortauthors}{\@nx\shorttitle}%
  \global\topskip42\p@\relax 
  \@settitle
  \ifx\@empty\authors \else \@setauthors \fi
  \ifx\@empty\@date \else {\vskip 1em \vtop{\centering\large\@date\@@par}}\fi
  \ifx\@empty\@dedicatory
  \else
    \baselineskip18\p@
    \vtop{\centering{\footnotesize\itshape\@dedicatory\@@par}%
      \global\dimen@i\prevdepth}\prevdepth\dimen@i
  \fi
  \@setabstract
  \normalsize
  \if@titlepage
    \newpage
  \else
    \dimen@34\p@ \advance\dimen@-\baselineskip
    \vskip\dimen@\relax
  \fi
} 
\renewcommand*\@adminfootnotes{%
  \let\@makefnmark\relax  \let\@thefnmark\relax
  \ifx\@empty\@subjclass\else \@footnotetext{\@setsubjclass}\fi
  \ifx\@empty\@keywords\else \@footnotetext{\@setkeywords}\fi
  \ifx\@empty\thankses\else \@footnotetext{%
    \def\par{\let\par\@par}\@setthanks}%
  \fi
}
\makeatother
 
\usepackage{fancyhdr}
\newcommand{\isEquivTo}[1]{\underset{#1}{\sim}}

\usepackage{lastpage}

\usepackage{comment}
\def\R{\mathbb R}
\def\z{\mathbb Z}

\def\HH{\mathcal H}

\def\AA{\mathcal A}

\def\la {{\lambda}}

\newcommand {\nc}   {\newcommand}
\nc {\be}   {\begin{equation}} \nc {\ee}   {\end{equation}} \nc
{\beq}  {\begin{eqnarray}} \nc {\eeq}  {\end{eqnarray}} \nc {\beqs}
{\begin{eqnarray*}} \nc {\eeqs} {\end{eqnarray*}}
\def\edc{\end{document}}

\usepackage{fancyhdr}
\providecommand{\abs}[1]{\lvert#1\rvert}

\newtheorem{theoreme}{Theorem}[section]
\newtheorem{pro}[theoreme]{Proposition}
\newtheorem{lemma}[theoreme]{Lemma}

\newtheorem{definition}[theoreme]{Definition}
\newtheorem{rem}[theoreme]{Remark}
\def\nline{\\ \noalign{\medskip}}
\numberwithin{equation}{section}
\renewenvironment{proof}{{\bfseries \noindent Proof.}}{\demo}
\newcommand\xqed[1]{%
  \leavevmode\unskip\penalty9999 \hbox{}\nobreak\hfill
  \quad\hbox{#1}}
\newcommand\demo{\xqed{$\square$}}
\hypersetup{bookmarks, colorlinks, urlcolor=blue, citecolor=blue, linkcolor=blue, hyperfigures, pagebackref,   pdfcreator=LaTeX, breaklinks=true, pdfpagelayout=SinglePage, bookmarksopen=true,bookmarksopenlevel=2}

\usepackage{fancyhdr}
 \renewenvironment{proof}{{\bfseries \noindent Proof.}}{\demo}

\hypersetup{bookmarks, colorlinks, urlcolor=blue, citecolor=red, linkcolor=blue, hyperfigures, pagebackref,
    pdfcreator=LaTeX, breaklinks=true, pdfpagelayout=SinglePage, bookmarksopen=true,bookmarksopenlevel=2}

\usepackage{lastpage}
\title[{\fontsize{6.5}{8}\selectfont Indirect Stability of a multidimensional coupled wave equations with one locally boundary fractional damping}]{Indirect Stability of a multidimensional coupled wave equations  with one locally boundary fractional damping}
 
\begin{document}
\author{Mohammad AKIL$^1$ and Ali Wehbe$^2$  \vspace{0.58cm}\\
$^1$Universit\'e Savoie Mont Blanc, Laboratoire LAMA, Chamb\'ery-France.\\
$^2$Lebanese University, Faculty of sciences 1, Khawarizmi Laboratory of Mathematics and Applications-KALMA, Hadath-Beirut. \\ \vspace{0.2cm}
Email: mohammad.akil@univ-smb.fr,  ali.wehbe@ul.edu.lb.
}
\begin{abstract}
In this work, we consider a system of  multidimensional   wave equations coupled by velocities  with one localized fractional  boundary damping. First, using a general criteria of Arendt- Batty,  by assuming that the boundary control region satisfy some geometric conditions,   under the equality speed propagation  and the coupling parameter of the two equations is small enough,  we show the strong stability of our system in the absence of the compactness of the resolvent. Our system is not uniformly stable in general since it is the case of the interval. Hence, we look for a  polynomial decay rate for smooth initial data for our system by applying a frequency domain approach combining with a multiplier method. Indeed, by assuming that the boundary control region satisfy some geometric conditions and the waves propagate with equal speed and the coupling parameter term is small enough, we establish a polynomial energy decay rate for smooth solutions, which depends on the order of the fractional derivative.
\end{abstract}

\maketitle
\pagenumbering{roman}
\maketitle
\tableofcontents
\pagenumbering{arabic}
\setcounter{page}{1}

\maketitle

\setcounter{equation}{0}
	\section{Introduction}
\noindent Let $\Omega$ be a bounded domain of $\R^d$, $d\geq 2$, with a Lipschitz boundary $\Gamma=\Gamma_0\cup \Gamma_1$, with $\Gamma_0$ and $\Gamma_1$ open subsets of $\Gamma$ such that $\overline{\Gamma_0}\cap \overline{\Gamma_1}=\emptyset$ and $\Gamma_1$ is non empty. We consider the multidimensional coupled wave equations
	\begin{eqnarray}
	u_{tt}-\Delta u+by_t&=&0,\quad \text{in}\quad \Omega\times (0,+\infty)\label{chap3eq1} ,\\
	y_{tt}-a\Delta y-bu_t&=&0,\quad \text{in}\quad \Omega\times (0,+\infty)\label{chap3eq2} ,\\
	u&=&0,\quad \text{on}\quad \Gamma_0\times (0,+\infty) ,\label{chap3eq3}\\
	y&=&0,\quad \text{on}\quad \Gamma\times (0,+\infty) ,\label{chap3eq4}\\
	\frac{\partial u}{\partial\nu}+\gamma\partial_t^{\alpha,\eta}u&=&0,\quad \text{on}\quad \Gamma_1\times (0,+\infty) ,\label{chap3eq5}
	\end{eqnarray}
	where $\nu$ is the unit outward normal vector along the boundary $\Gamma_1$, $\gamma$ is a positive constant involved in the boundary control, $a>0$ and $b\in \R_{\ast}$. The notation $\partial_t^{\alpha,\eta}$ stands the generalized Caputo's fractional derivative see \cite{Caputo:76} of order $\alpha$ with respect to the time variable and is defined by 
	\begin{equation}\label{CaputoDerivative}
	[D^{\alpha,\eta}\omega](t)=\partial_t^{\alpha,\eta}\omega(t)=\frac{1}{\Gamma(1-\alpha)}\int_0^t(t-s)^{-\alpha}e^{-\eta(t-s)}\frac{d\omega}{ds}(s)ds,\ \ 0<\alpha<1,\ \eta\geq 0.
	\end{equation}
	The system \eqref{chap3eq1}-\eqref{chap3eq5} is considered with initial conditions 
	\begin{eqnarray}
	u(x,0)=u_0(x),&u_t(x,0)=u_1(x)&\text{where}\quad x\in \Omega,\label{chap3eq6}\\
	y(x,0)=y_0(x),&y_t(x,0)=y_1(x)&\text{where}\quad x\in \Omega,\label{chap3eq7}
	\end{eqnarray}

	\noindent The fractional derivative operator of order $\alpha$, $0<\alpha<1$, is defined by 
	\begin{equation}\label{N4}
	[D^{\alpha}f](t)=\int_0^t\dfrac{(t-\tau)^{-\alpha}}{\Gamma(1-\alpha)}\dfrac{df}{d\tau}(\tau)d\tau.
	\end{equation}
	The boundary fractional damping of the type   $\partial_t^{\alpha,\eta}u$ where $0<\alpha<1$, $\eta\geq 0$ arising from the material property has been used in several applications such as in physical, chemical, biological, ecological phenomena. For more details we refer the readers to  \cite{Mbodje:06},  \cite{mbomon:95},  \cite{bagleytorvik2:83}, \cite{bagleytorvik3:83}, \cite{bagleytorvik1:83} and \cite{mainardibonetti}. In theoretical point of view, fractional derivatives involves singular and non-integrable kernels ($t^{-\alpha}$, $0<\alpha<1$). This leads to substantial mathematical difficulties since all the previous methods developed for convolution terms with regular and/or integrable kernels are no longer valid. \\
	There are a few numbers of publications concerning the stabilization of a distributed system with fractional damping. In \cite{Mbodje:06}, B. Mbodje considered a $1-d$  wave equation with boundary fractional damping acting on a part of the boundary of the domain:
	\begin{equation}\label{I2}
	\left\{\begin{array}{rllll}
	\partial_t^2u(x,t)-\partial_x^2u(x,t)&=&0,&0<x<1,&t>0,\\
	u(0,t)&=&0.&&\\
	\partial_xu(1,t)&=&-\gamma\partial_t^{\alpha,\eta}u(1,t),&0<\alpha<1,&\eta\geq0,\\
	u(x,0)&=&u_0(x),&&\\
	\partial_tu(x,0)&=&v_0(x),
	\end{array}
	\right.
	\end{equation}
	Firstly, he proved that system \eqref{I2} is not uniformly stable, on other words its energy has no exponential decay rate. However, using LaSalle's invariance principle, he proved that system \eqref{I2} is strongly stable for the usual initial data. Secondly, he established a polynomial energy decay rate of type $\frac{1}{t}$ for smooth initial data. In \cite{akilwehbe01}, Akil and wehbe considered a multidimensional wave equation with boundary fractional damping acting on a part of the boundary of the domain:
	\begin{equation}\label{AW}
	\left\{\begin{array}{lllll}
	u_{tt}-\Delta u&=&0,&\text{in}&\Omega\times \R^{+},\\
	u&=&0,&\text{on}&\Gamma_0\times \R^{+},\\
	\displaystyle
	\frac{\partial u}{\partial\nu}+\gamma\partial_t^{\alpha,\eta}u&=&0,&\text{on}&\Gamma_1\times \R^+,\\
	u(x,0)&=&u_0(x)&\text{in}&\Omega,\\
	u_t(x,0)&=&u_1(x),&\text{in}&\Omega.
	\end{array}
	\right.
	\end{equation}
	Firstly, combining  general criteria of Arendt and Batty with Holmgren's theorem we showed the strong stability of system \eqref{AW} in the absence of the compactness of the resolvent and without any additional geometric conditions. Next, the authors showed that their system is not uniformly stable in general, since it is the case of the interval. Hence, we look for a polynomial decay rate for smooth initial data for our system by applying a frequency domain approach combining with a multiplier method. Indeed, by assuming that the boundary control region satisfy the Geometric Control Condition \textbf{(GCC)} and by
	using the exponential decay of the wave equation with a standard damping
	$$
	\partial_{\nu}u(x,t)+u_t(x,t)=0,\quad \text{on}\quad \gamma_1\times \R_+^{\ast}
	$$
	they established a polynomial energy decay rate for smooth solutions, which depends on the order of the fractional derivative ($t^{-\frac{1}{1-\alpha}}$).
	In \cite{zhda:14}, Zhang and Dai considered the multidimensional wave equation with boundary source term and fractional dissipation defined by 
	\begin{equation}\label{I3}
	\left\{\begin{array}{rllll}
	u_{tt}-\Delta u&=&0,&x\in \Omega&t>0,\\
	u&=&0,&x\in\Gamma_0&t>0,\\
	\displaystyle{\frac{\partial u}{\partial\mu}+\partial_t^{\alpha}u}&=&|u|^{m-1}u,&x\in \Gamma_1&t>0,\\
	u(x,0)&=&u_0,&x\in \Omega,&\\
	u_t(x,0)&=&u_1(x),&x\in \Omega&
	\end{array}
	\right.
	\end{equation}
	where $m>1$. They proved by Fourier transforms and the Hardy-Littlewood-Sobolev inequality the exponential stability for sufficiently large initial data.	\\
	In \cite{benaissa:17}, Benaissa and al. considered the Euler-Bernoulli beam equation with boundary dissipation of fractional derivative type defined by 
	\begin{equation}\label{I4}
	\left\{\begin{array}{rllll}
	\varphi_{tt}(x,t)+\varphi_{xxxx}(x,t)&=&0,&\text{in}&]0,L[\times ]0,+\infty[,\\
	\varphi(0,t)=\varphi_x(0,t)&=&0,&\text{in}&]0,+\infty[,\\
	\varphi_{xx}(L,t)&=&0,&\text{in}&]0,+\infty,\\
	\varphi_{xxx}(L,t)&=&\gamma\partial_t^{\alpha,\eta}\varphi(L,t),&\text{in}&]0,+\infty[
	\end{array}\right.
	\end{equation}
	where $0<\alpha<1$, $\eta\geq 0$ and $\gamma>0$. If $\eta=0$,  by  using spectral analysis, they proved the non-uniform stability. Otherwise, if  $\eta> 0$, they  proved that the energy of system \eqref{I4} decay as time goes to infinity as $\frac{1}{t^{\frac{1}{1-\alpha}}}$.\\
	In \cite{Boussouira01} see also (\cite{Boussouira03}-\cite{Boussouira02}), Alabau-Boussouira studied the boundary indirect stabilization of a system of two-level second order evolution equations coupled through the zero-order terms. The lack of uniform stability as proved in the case where the ratio of the wave propagation speeds of the two equation is equal to $\frac{1}{k^2}$ with $k$ being an integer and $\Omega$ is a cubic domain in $\R^3$, or by a compact perturbation argument and a polynomial energy decay rate of type $\frac{1}{\sqrt{t}}$ is obtained by a general integral inequality in the case where the wave propagates at the same speed and $\Omega$ is a star-shaped domain in $\R^N$. These results are very interesting but not optimal.\\ 
	In \cite{AmmariMehre01}, Ammari and Mehrenberger gave a characterization of the stability of a system of two evolution equations coupling through the velocity terms subject one bounded viscous feedbacks.\\ 
	In \cite{LiuRao01} Liu and Rao, considered a system of two coupled wave equations with one boundary damping  described by 
	\begin{equation}\label{LR1}
	\left\{\begin{array}{lllll}
	u_{tt}-a\Delta u+\alpha y&=&0,&\text{in}&\Omega\times \R^+,\\
	y_{tt}-\Delta y+\alpha u&=&0,&\text{in}&\Omega\times \R^+,\\
	a\partial_{\nu}u+\gamma u+u_t&=&0,&\text{on}&\Gamma_1\times \R^+,\\
	u&=&0,&\text{on}&\Gamma_0\times \R^+,\\
	y&=&0,&\text{on}&\Gamma\times \R^{+}
	\end{array}
	\right.
	\end{equation}
	Under some arithmetic condition on the ratio of the wave propagation speeds of the two equations, they established  a polynomial energy decay rate for smooth initial data on a 1-dimensional domain. Furthermore, under the equality speed wave propagation, they proved that the energy of the system \eqref{LR1} decays at the rate $\frac{1}{t}$ for smooth initial data on an $N-$dimensional domain $\Omega$ with the usual geometrical condition.

	\noindent In \cite{KodjaBader01} Ammar-Khodja and Bader studied the simultaneous boundary stabilization of a system of two wave equations coupling through the velocity terms described by  
	\begin{equation}\label{AMMKH}
	\left\{\begin{array}{lllll}
	u_{tt}-u_{xx}+b(x)y_t&=&0,&\text{in}&(0,1)\times (0,+\infty),\\
	y_{tt}-ay_{xx}-b(x)u_t&=&0,&\text{in}&(0,1)\times (0,+\infty),\\
	y_t(0,t)-\alpha\left(y_x(0,t)+u_t(0,t)\right)&=&0,&\text{in}&(0,+\infty),\\
	u_x(0,t)-\alpha y_t(0,t)&=&0,&\text{in}&(0,+\infty),\\
	u(1,t)=y(1,t)&=&0,&\text{in}&(0,+\infty)
	\end{array}
	\right.
	\end{equation}
	where $a$ and $\alpha$ are two constants strictly positives and $b\in {\mathcal{C}}^0\left([0,1]\right)$.  In the general case, when $a\neq 1$, they proved that the system \eqref{AMMKH} is uniformly stable if and only if it is strongly stable and there exists  integer numbers $p$ and $q$ such that $a=\frac{(2p+1)^2}{q^2}$.  Otherwise, under the equal speed wave propagation condition (i.e.  $a=1$), they proved that the system \eqref{AMMKH} is uniformly stable if and only if it is strongly stable and the coupling parameter $b(x)$ verifies that $\int_0^1b(x)dx\neq \frac{(2k+1)\pi}{2}$ for any $k\in \z$. Note that, system \eqref{AMMKH} is damped by two related boundary controls.\\
In \cite{akilwehbe02}, Akil, Ghader and Wehbe considered a $1-d$ coupled wave equations on its indirect boundary stabilization defined by: 
\begin{equation}\label{AGW}
	\left\{\begin{array}{lllll}
	u_{tt}-u_{xx}+by_t&=&0,&\text{in}&(0,1)\times \R^{\ast}_+,\\
	y_{tt}-ay_{xx}-bu_t&=&0,&\text{in}&(0,1)\times \R^{\ast}_+,\\
	u_x(1,t)+\gamma \partial_t^{\alpha,\eta}y(1,t)&=&0,&\text{in}&\R^{\ast}_+,\\
	u(0,t)=y(0,t)=y(1,t)&=&0,&\text{in}&\R^{\ast}_+,
	\end{array}\right.
	\end{equation}
	where $a>0$ and $b\in \R^{\ast}$. Firstly, they proved that system \eqref{AGW} is strongly stable if and only if the coupling parameter $b$ is outside a discrete set $S$ of exceptional values. Next, for $b\notin S$, they proved that the energy decay rate of system \eqref{AGW} is greatly influenced by the nature of the coupling parameter $b$ (an additional condition on $b$) and by the arithmetic property of the ratio of the wave propagation speeds $a$. They established a polynomial energy decay rate of type $t^{-s(\alpha)}$, such that 
\begin{enumerate}
\item If $a=1$ and $b\neq k\pi$, then $s(\alpha)=\frac{2}{1-\alpha}$.
\item If ($a=1$ and $b=k\pi$), or ($a\neq 1$, $a\in \mathbb{Q}$, $\sqrt{a}\notin \mathbb{Q}$ and $b$ small enough), or ($a\neq 1$ and $\sqrt{a}\in \mathbb{Q}$) or ($a\neq 1$ and for almost $\sqrt{a}\in \mathbb{R}/\mathbb{Q}$), then $s(\alpha)=\frac{2}{5-\alpha}$.
\end{enumerate}
\noindent The polynomial energy decay rate occurs in many control problems where the open-loop systems are strongly stable, but not exponentially stable (see \cite{Lasiescka02}). We quote \cite{Lebau01}, \cite{Lebau02} for wave equations with local internal or boundary damping, \cite{batkai:06} and \cite{liurao:05} for abstract system, \cite{Russell01} and \cite{ZuazuazZhang01} for systems of coupled wave-heat equations.\\[0.1in]
	\noindent This paper is organized as follows: In Subsection \ref{AMWP}, we reformulate the system \eqref{chap3eq1}-\eqref{chap3eq7} into an augmented model system by coupling the wave equation with a suitable equation and we prove the well-posedness of our system by semigroup approach. In subsection \ref{strongstab}, under the equal speed wave propagation condition (i.e. $a=1$) and if the coupling parameter $b$ is small enough, using a general criteria of Arendt-Batty  theorem, we show that the strong stability of our system for  in the absence of the compactness of the resolvent and under the multiplier geometric control condition noted by $\textbf{(MGC)}$. 
 In Section \ref{PSMCC}, under the equal speed wave propagation and the coupling parameter $b$ verify another condition, we look for  a polynomial decay rate for smooth initial data for our system by applying a frequency domain approach combining with a multiplier method. Indeed, by assuming that the boundary control region satisfy $\textbf{(MGC)}$ boundary condition, we establish a   polynomial energy decay for smooth solution of type $\frac{1}{t^{\frac{2}{1-\alpha}}}$.
\section{Well-Posedness and Strong Stability}\label{chap3WPSS}
\noindent In this section, we will study the strong stability of system \eqref{chap3eq1}-\eqref{chap3eq7} in the absence of the compactness of the resolvent and by the $\textbf{(MGC)}$ condition.  First, we will study the existence, uniqueness, and regularity of the system of our system.
\subsection{Augmented model and well-Posedness}\label{AMWP}
\noindent Firstly, we reformulate system \eqref{chap3eq1}-\eqref{chap3eq7}. For this aim, we use the following Theorem 
\begin{theoreme}\label{reformulated}\rm{(See \cite{akilwehbe01})
Let $\alpha\in (0,1)$ and $\mu$ be the function defined almost everywhere on $\R$ by 
\begin{equation*}
\mu(\xi)=|\xi|^{\frac{2\alpha-d}{2}}.
\end{equation*}
Then, for $\eta\geq 0$, the  relation between the "input" $U$ and the "output" $O$ of the following system 
\begin{equation*}
\begin{array}{lllll}
\partial_t\omega(x,t,\xi)+\left(\abs{\xi}^2+\eta\right)\omega(x,t,\xi)-U(x,t)\mu(\xi)&=&0,\quad (x,t,\xi)\in \Omega\times (0,+\infty)\times \R^d,\nline
\omega(x,0,\xi)&=&0,\quad (x,\xi)\in \Omega\times \R^d,\nline
\displaystyle
O(x,t)-\kappa(\alpha,d)\int_{\R^d}\mu(\xi)\omega(x,t,\xi)d\xi &=&0,\quad (x,t)\in \Omega\times (0,+\infty)
\end{array}
\end{equation*}
is given by 
\begin{equation}\label{App5}
O=I^{1-\alpha,\eta}U,
\end{equation}
where 
\begin{equation}\label{InverseCaputoDerivative}
\kappa(\alpha,d)=\frac{2\sin(\alpha\pi)\Gamma\left(\frac{d}{2}+1\right)}{d\pi^{\frac{d}{2}+1}}\quad \text{and}\quad [I^{\alpha,\eta}U](x,t)=\int_0^t\frac{(t-\tau)^{\alpha-1}e^{-\eta(t-\tau)}}{\Gamma(\alpha)}U(x,\tau)d\tau.
\end{equation}
From Equations \eqref{CaputoDerivative} and \eqref{InverseCaputoDerivative} it is clearly 
\begin{equation*}
D^{\alpha,\eta}U=I^{1-\alpha,\eta}DU.
\end{equation*}}
\end{theoreme}
\noindent From Theorem \ref{reformulated}, system  \eqref{chap3eq1}-\eqref{chap3eq7} may be recast into the following model:
\begin{eqnarray}
u_{tt}(x,t)-\Delta u(x,t)+by_t(x,t)&=&0,\quad (x,t)\in \Omega\times (0,+\infty),\label{aug1}\nline
y_{tt}(x,t)-a\Delta y(x,t)-bu_t(x,t)&=&0,\quad (x,t)\in \Omega\times (0,+\infty),\label{aug2}\nline
\partial_t\omega(x,t,\xi)+\left(|\xi|^2+\eta\right)\omega(x,t,\xi)-\mu(\xi)\partial_tu(x,t,\xi)&=&0,\quad (x,t,\xi)\in \Gamma_1\times (0,+\infty)\times \R^d \label{aug3},\nline
u(x,t)&=&0,\quad (x,t)\in \Gamma_0\times (0,+\infty) ,\label{aug4}\nline
y(x,t)&=&0,\quad (x,t)\in \Gamma\times (0,+\infty) ,\label{aug5}\nline 
\frac{\partial u}{\partial\nu}(x,t)+\gamma\kappa(\alpha,d) \int_{\R^d}\mu(\xi)\omega(x,t,\xi)d\xi &=&0,\quad  (x,t,\xi)\in \Gamma_1\times (0,+\infty)\times \R^d \label{aug6}
\end{eqnarray}
where $\gamma$ is a positive constant, $\eta\geq 0$ and $\kappa(\alpha,d)=\frac{2\sin(\alpha\pi)\Gamma\left(\frac{d}{2}+1\right)}{d\pi^{\frac{d}{2}+1}}$. Finally, system \eqref{aug1}-\eqref{aug6} is considered with the following initial conditions
\begin{eqnarray}
u(x,0)=u_0(x),&u_t(x,0)=u_1(x)&\text{where}\quad x\in \Omega,\label{aug7}\\
y(x,0)=y_0(x),&y_t(x,0)=y_1(x)&\text{where}\quad x\in \Omega,\label{aug8}\\
&\omega(x,0,\xi)=0&\hspace{0.08cm}\text{where}\quad (x,\xi)\in \Gamma_1\times \R^d.\label{aug9}
\end{eqnarray}
Let us define the energy space $\mathcal{H}$ by  
\begin{equation}\label{matcH}
\mathcal{H}=H_{\Gamma_0}^1(\Omega)\times L^2(\Omega)\times H_0^1(\Omega)\times L^2(\Omega)\times L^2(\Gamma_1\times\R^d),
\end{equation}
such that $H_{\Gamma_0}^1(\Omega)$ is given by 
$$
H_{\Gamma_0}^1(\Omega)=\left\{u\in H^1(\Omega),\quad u=0\ \ \text{on}\ \ \Gamma_0\right\}.
$$
It is easy to check that the spaces $H_{\Gamma_0}^1(\Omega)$ and  $L^2(\Gamma_1\times\R^d)$  are Hilbert spaces over $\mathbb{C}$ equipped respectively  with the norms
$$\left\|u\right\|_{H_{\Gamma_0}^1(\Omega)}=\left(\int_{\Omega}|\nabla u(x)|^2dx\right)^{\frac{1}{2}}\ \ \ \text{and}\ \ \  \left\|\omega\right\|_{ L^2(\Gamma_1\times\R^d)}=\left(\int_{\Gamma_1}\int_{\R^d}|\omega(x,\xi)|^2d\xi d\Gamma\right)^{\frac{1}{2}}. $$
The energy space $\mathcal{H}$ is equipped with the inner product defined by
$$
\left<U,\widetilde{U}\right>_{\mathcal{H}}=\int_{\Omega}\left(v\bar{\tilde{v}}+\nabla u\nabla{\bar{\tilde{u}}}+z\bar{\tilde{z}}+a\nabla y\nabla{\bar{\tilde{y}}}\right)dx+\gamma\kappa(\alpha,d)\int_{\Gamma_1}\int_{\R^d}\omega(x,\xi)\bar{\tilde{\omega}}(x,\xi)d\xi d\Gamma,
$$
for all    
$U=(u,v,y,z,\omega),\ \widetilde{U}=(\tilde{u},\tilde{v},\tilde{y},\tilde{z},\tilde{\omega})$ in $\mathcal{H}
$.  We use $\|U\|_{\mathcal{H}}$ to denote the corresponding norm.
The energy of   system \eqref{aug1}-\eqref{aug9} is given by 
\begin{equation}\label{ENERG}
E(t)=\frac{1}{2}\left\|(u,u_t,y,y_t,\omega)\right\|_{\mathcal{H}}^2.
\end{equation}
\begin{lemma}\label{DDENERG}
\rm{Let $U=\left(u,u_t,y,y_t,\omega\right)$ be a regular solution of problem \eqref{aug1}-\eqref{aug9}. Then, the functional energy defined in equation \eqref{ENERG} satisfies 
\begin{equation}\label{DENERG}
\displaystyle{\frac{d}{dt}E(t)=-\gamma\kappa(\alpha,d)\int_{\Gamma_1}\int_{\R^d}\left(|\xi|^2+\eta\right)\left|\omega(x,t,\xi)\right|^2 d\xi d\Gamma}.
\end{equation}}
\end{lemma}
\begin{proof}
Firstly, multiplying equations \eqref{aug1}  and \eqref{aug2} by $\bar{u_t}$ and $\bar{y_t}$ respectively, then using integration by parts over $\Omega$ and tacking the real part, we get 
\begin{equation}\label{DENERG1}
\displaystyle{
\frac{1}{2}\frac{d}{dt}\left(\int_{\Omega}\left(|u_t|^2+|\nabla u|^2+|y_t|^2+a|\nabla y|^2dx\right)\right)+\gamma\kappa(\alpha,d)\Re\left(\int_{\Gamma_1}\bar{u}_t\int_{\R^d}\mu(\xi)\omega(x,t,\xi)d\xi d\Gamma\right)=0,
}
\end{equation}
where $\Re$ stands for the real part of a complex number.
Secondly, multiplying equation \eqref{aug3} by $\gamma\kappa(\alpha,d)\bar{\omega}(x,\xi,t)$, then  integrating over $\Gamma_1\times \R^d$, we get 
\begin{equation}\label{DENERG2}
\begin{split}
\frac{1}{2}\frac{d}{dt}\left(\gamma\kappa(\alpha,d)\int_{\Gamma_1}\int_{\R^d}|\omega(x,t,\xi)|^2d\xi d\Gamma\right)+\gamma\kappa(\alpha,d)\int_{\Gamma_1}\int_{\R^d}\left(|\xi|^2+\eta\right)|\omega(x,t,\xi)|^2d\xi d\Gamma=\\
\gamma\kappa(\alpha,d)\Re\left(\int_{\Gamma_1}u_t\int_{\R^d}\mu(\xi)\bar{\omega}(x,t,\xi)d\xi d\Gamma\right).
\end{split}
\end{equation}
Combining equations \eqref{DENERG1}-\eqref{DENERG2}, we obtain 
$$
\displaystyle{
\frac{d}{dt}E(t)=-\gamma\kappa(\alpha,d)\int_{\Gamma_1}\int_{\R^d}\left(|\xi|^2+\eta\right)|\omega(x,t,\xi)|^2d\xi d\Gamma.
}
$$  
\end{proof}

\noindent From Lemma \eqref{DDENERG}, we deduce that the system \eqref{aug1}-\eqref{aug9} is dissipative in the sense that its energy is a non-increasing function  with respect to  the time variable $t$. We now define the linear unbounded operator $\mathcal{A}$ by 
\begin{equation}\label{domaine}
D(\mathcal{A})=\left\{
\begin{array}{l}
U=(u,v,y,z,\omega)\in\mathcal{H};\ \Delta u\in L^2(\Omega),\ \ \ y\in H^2(\Omega)\cap H_0^1(\Omega),\\
\vspace{0.15cm}v\in H_{\Gamma_0}^1(\Omega),\ \ z\in H_0^1(\Omega),\ \  -(|\xi|^2+\eta)\omega+v|_{\Gamma_1}\mu(\xi)\in L^2(\R^d\times \Gamma_1),\\
\vspace{0.15cm}\dfrac{\partial u}{\partial\nu}+\displaystyle{\gamma\kappa\int_{\R^d}\mu(\xi)\omega(x,\xi)d\xi=0}\ \ \text{on}\ \ \Gamma_1,\ \ \ |\xi|\omega\in L^2(\R^d\times \Gamma_1)
\end{array}\right\}
\end{equation}
and
$$
\mathcal{A}U=\left(v,\ \Delta u-bz,\ z,\ a\Delta y+bv,\ -\left(\abs{\xi}^2+\eta\right)+v|_{\Gamma_1}\mu(\xi)\right),\quad \forall\ U=(u,v,y,z,\omega)\in D(\mathcal{A}).
$$
If $U=\left(u,u_t,y,y_t,\omega\right)$ is the state of \eqref{aug1}-\eqref{aug6}, then this system can be transformed into the first order evolution equation on the Hilbert space $\mathcal{H}$:
\begin{equation}\label{evolution}
U_t(x,t,\xi)=\mathcal{A}U(x,t,\xi),\quad U(x,0)=U_0(x),
\end{equation}
where 
$$
U_0(x)=\left(u_0(x),u_1(x),y_0(x),y_1(x),0\right).
$$
\begin{pro}\label{Maximal}
{\rm{For $\eta\geq0$, the unbounded linear operator $\mathcal{A}$ is m-dissipative in the energy space $\mathcal{H}$.}}
\end{pro}
\begin{proof}
For all $U=(u,v,y,z,\omega)\in D(\mathcal{A})$, one has  
\begin{equation}\label{dissipatnes}
\Re\left(\left<\mathcal{A}U,U\right>_{\mathcal{H}}\right)=-\gamma\kappa(\alpha,d)\int_{\Gamma_1}\int_{\R^d}(|\xi|^2+\eta)|\omega(x,\xi)|^2d\xi d\Gamma\leq0.
\end{equation}
\noindent which implies that $\mathcal{A}$ is dissipative. Now, let $F=\left(f_1,f_2,f_3,f_4,f_5\right)\in \mathcal{H}$, we show the existence  of $U=(u,v,y,z,\omega)\in D(\mathcal{A})$, unique solution of the following  equation
\begin{equation}\label{W1}
\left(I-\mathcal{A}\right)U=F.
\end{equation}
Equivalently, one must consider the system given by 
\begin{eqnarray}
u-v&=&f_1,\label{W2}\\
v-\Delta u+bz&=&f_2,\label{W3}\\
y-z&=&f_3,\label{W4}\\
z-(a\Delta y+bv)&=&f_4,\label{W5}\\
\omega(x,\xi)+(|\xi|^2+\eta)\, \omega(x,\xi)-v|_{\Gamma_1}\mu(\xi)&=&f_5(x,\xi)\label{W6}.
\end{eqnarray}
From \eqref{W2} and \eqref{W6}, we get  
\begin{equation}\label{W7}
\omega(x,\xi)=\frac{f_5(x,\xi)}{1+\abs{\xi}^2+\eta}+\frac{u|_{\Gamma_1}\mu(\xi)}{1+\abs{\xi}^2+\eta}-\frac{f_1(x)\, \mu(\xi)}{1+\abs{\xi}^2+\eta}.
\end{equation}
Inserting \eqref{W2} and \eqref{W4} in \eqref{W3} and \eqref{W5}, we get 
\begin{eqnarray}
u-\Delta u+by&=&f_1+f_2+bf_3,\label{W8}\\
y-a\Delta y-bu&=&f_3+f_4-bf_1,\label{W9}
\end{eqnarray}
with the boundary conditions 
\begin{equation}\label{W10}
u=0\quad \text{on}\quad \Gamma_0,\quad \dfrac{\partial u}{\partial\nu}=-\gamma \kappa(\alpha,d)\int_{\R^d}\mu(\xi)\omega(x,\xi)d\xi\quad\text{on}\quad \Gamma_1\quad \text{and}\quad y=0\quad\text{on}\quad \Gamma.
\end{equation}
Let $\phi=(\varphi,\psi)\in H^1_{\Gamma_0}(\Omega)\times H_0^1(\Omega)$. Multiplying Equations \eqref{W8} and \eqref{W9} by $\bar{\varphi}$ and $\bar{\psi}$ respectively, integrate over $\Omega$, then using by parts integration, we get
\begin{eqnarray}
\int_{\Omega}u\bar{\varphi}dx+\int_{\Omega}\nabla u\nabla\bar{\varphi}dx-\int_{\Gamma_1}\frac{\partial u}{\partial\nu}\bar{\varphi}d\Gamma+b\int_{\Omega}y\bar{\varphi}dx&=&\int_{\Omega}F_1\bar{\varphi}dx,\label{W11}\\
\int_{\Omega}y\bar{\psi}dx+a\int_{\Omega}\nabla y\nabla \bar{\psi}dx-b\int_{\Omega}u\bar{\psi}dx&=&\int_{\Omega}F_2\bar{\psi}dx,\label{W12}
\end{eqnarray}
where $F_1= f_1+f_2+bf_3$ and $F_2=f_3+f_4-bf_1$. Using equations \eqref{W7} and \eqref{W10}, we get 
\begin{equation}\label{w11}
-\int_{\Gamma_1}\frac{\partial u}{\partial\nu}\bar{\varphi}d\Gamma=M_1\left(\alpha,\eta\right)+M_2\left(\alpha,\eta\right)\int_{\Gamma_1}u\bar{\varphi}d\Gamma-M_2(\alpha,\eta)\int_{\Gamma_1}f_1\bar{\varphi},
\end{equation}
where
\begin{equation}\label{M1andM2}
M_1(\alpha,\eta)=\gamma\kappa(\alpha,d)\int_{\Gamma_1}\int_{\R^d}\frac{\mu(\xi)f_5(x,\xi)\bar{\varphi}}{1+\abs{\xi}^2+\eta}d\xi d\Gamma\quad \text{and}\quad M_2(\alpha,\eta)=\gamma\kappa(\alpha,d)\int_{\R^d}\frac{\mu^2(\xi)}{|\xi|^2+\eta+1}d\xi.
\end{equation} 
From Lemma \ref{Appendix1}, we have $$M_2(\alpha,d)=\gamma (1+\eta)^{\alpha-1}.$$
Next, by using Cauchy-Schwartz inequality, we get 
\begin{equation}\label{absM1}
\abs{M_1\left(\alpha,\eta\right)}\leq \gamma\kappa(\alpha,d)\left(\int_{\R^d}\frac{\mu^2(\xi)}{\left(1+\abs{\xi}^2+\eta\right)^2}d\xi\right)^{\frac{1}{2}}\|\varphi\|_{L^2(\Gamma_1)}\|f_5\|_{L^2(\Gamma_1\times \R^d)}.
\end{equation}
On the other hand, since $\alpha\in (0,1)$ and $\eta\geq 0$, we get
\begin{equation}\label{otb}
\displaystyle{\int_{\R^d}\frac{\mu^2(\xi)}{\left(1+\abs{\xi}^2+\eta\right)^2}d\xi<\int_{\R^d}\frac{\mu^2(\xi)}{1+\abs{\xi}^2+\eta}d\xi\leq\frac{M_2(\alpha,d)}{\gamma\kappa(\alpha,d)}}\leq \frac{(1+\eta)^{\alpha-1}}{\kappa(\alpha,d)}<+\infty. 
\end{equation}
Inserting \eqref{otb} in \eqref{absM1}, then using the fact that $f_5\in L^2(\Gamma_1\times \R^d)$ and 
the trace theorem, we get that $M_1(\alpha,\eta)$ is well defined. 
Adding Equations \eqref{W11} and \eqref{W12}, we obtain 
\begin{equation}\label{W13}
a((u,y),(\varphi,\psi))=L(\varphi,\psi),\quad \forall (\varphi,\psi)\in H_{\Gamma_0}^1(\Omega)\times H_0^1(\Omega),
\end{equation}
where 
\begin{equation}\label{a}
\begin{array}{lll}
a((u,y),(\varphi,\psi))&=&\displaystyle
\int_{\Omega}u\bar{\varphi} dx+\int_{\Omega}\nabla u\nabla\bar{\varphi} dx+\int_{\Omega}y\bar{\psi} dx+a\int_{\Omega}\nabla y\nabla\bar{\psi} dx\\
\vspace{0.15cm}&&\displaystyle
+M_2(\alpha,\eta)\int_{\Gamma_1}u\bar{\varphi} d\Gamma+b\int_{\Omega}y\bar{\varphi}dx-b\int_{\Omega}u\bar{\psi}dx
\end{array}
\end{equation}
and 
\begin{equation}\label{L}
L(\varphi,\psi)=
\displaystyle
\int_{\Omega}F_1\bar{\varphi}dx+\int_{\Omega}F_2\psi dx-M_1(\alpha,\eta)+M_2\int_{\Gamma_1}f_1\varphi d\Gamma.
\end{equation}
Thanks to \eqref{absM1}, \eqref{a}, \eqref{L} and using the fact that $M_2(\alpha,\eta)>0$, we have that $a$ is a bilinear continuous coercive form on $\left(H_{\Gamma_0}^1(\Omega)\times H_0^1(\Omega)\right)^2$, and $L$ is Linear continuous form on $H_{\Gamma_0}^1(\Omega)\times H_0^1(\Omega)$. Then, using Lax-Milgram Theorem, we deduce that there exists $(u,y)\in H_{\Gamma_0}^1(\Omega)\times H_0^1(\Omega)$ unique solution of the variational problem \eqref{W13}. Applying the classical elliptic regularity and by choosing appropriate test function in \eqref{W13}, we deduce that system \eqref{W8}-\eqref{W10} has a unique solution  $(u,y)\in H_{\Gamma_0}^1(\Omega)\times \left(H^2(\Omega)\cap H_0^1(\Omega)\right)$ such that $\Delta u\in L^2(\Omega)$. Define 
\begin{equation}\label{WWW1}
v:=u-f_1,\ z:=y-f_3,\quad \text{and}\quad  \omega(x,\xi):=\frac{f_5(x,\xi)}{1+\abs{\xi}^2+\eta}+\frac{u|_{\Gamma_1}\mu(\xi)}{1+\abs{\xi}^2+\eta}-\frac{f_1|_{\Gamma_1}\mu(\xi)}{1+\abs{\xi}^2+\eta}.
\end{equation}
In order to complete the existence of $U$ in $D\left(\mathcal{A}\right)$, we need to prove $\omega(x,\xi)$ and $\abs{\xi}\omega(x,\xi)\in L^2\left(\Gamma_1\times \R^{d}\right)$. From \eqref{WWW1}, we get 
\begin{equation*}
\int_{\Gamma_1}\int_{\R^d}\abs{\omega(x,\xi)}^2d\xi d\Gamma_1\leq 3\int_{\Gamma_1}\int_{\R^d}\frac{\abs{f_5(x,\xi)}^2}{(1+\abs{\xi}^2+\eta)^2}d\xi d\Gamma_1+3\left(\int_{\Gamma_1}(\abs{u}^2+\abs{f_1}^2)d\Gamma_1\right)\int_{\R^d}\frac{\abs{\xi}^{2\alpha-d}}{(1+\abs{\xi}^2+\eta)^2}d\xi.
\end{equation*}
Using Lemma \ref{Appendix1}, it easy to see that 
$$
\int_{\R^d}\frac{\abs{\xi}^{2\alpha-d}}{(1+\abs{\xi}^2+\eta)^2}d\xi \leq \int_{\R^d}\frac{\abs{\xi}^{2\alpha-d}}{1+\abs{\xi}^2+\eta}d\xi<+\infty.
$$
On the other hand, using the fact that $f_5\in L^2\left(\Gamma_1\times \R^d\right)$, we obtain 
$$
\int_{\Gamma_1}\int_{\R^d}\frac{\abs{f_5(x,\xi)}^2}{(1+\abs{\xi}^2+\eta)^2}d\xi d\Gamma_1\leq \frac{1}{(1+\eta)^2}\int_{\Gamma_1}\int_{\R^d}\abs{f_5(x,\xi)}^2d\xi d\Gamma_1<+\infty.
$$
Hence, using trace theorem we obtain $\omega(x,\xi)\in L^2(\Gamma_1\times \R^d)$. Next, using \eqref{WWW1}, we get 
$$
\int_{\Gamma_1}\int_{\R^d}\abs{\xi\omega}^2d\xi d\Gamma_1\leq 3\int_{\Gamma_1}\int_{\R^d}\frac{\abs{\xi f_5}^2}{(1+\abs{\xi}^2+\eta)^2}d\xi d\Gamma_1+3\left(\int_{\Gamma_1}(\abs{u}^2+\abs{f_1}^2)d\Gamma_1\right)\int_{\R^d}\frac{\abs{\xi}^{2\alpha-d+2}}{(1+\abs{\xi}^2+\eta)^2}d\xi.
$$
Using trace theorem and Lemma \ref{Appendix11}, we get
$$
\left(\int_{\Gamma_1}(\abs{u}^2+\abs{f_1}^2)d\Gamma_1\right)\int_{\R^d}\frac{\abs{\xi}^{2\alpha-d+2}}{(1+\abs{\xi}^2+\eta)^2}d\xi<+\infty. 
$$
Using the facts that $\abs{\xi}^2<\left(\abs{\xi}^2+\eta+1\right)^2$ and $f_5\in L^2(\Gamma_1\times \R^d)$, we get 
$$
\int_{\Gamma_1}\int_{\R^d}\frac{\abs{\xi f_5}^2}{(1+\abs{\xi}^2+\eta)^2}d\xi d\Gamma_1\leq \int_{\Gamma_1}\int_{\R^d}\abs{f_5}^2d\xi d\Gamma_1<+\infty
$$
Hence we obtain, $\abs{\xi \omega}\in L^2\left(\Gamma_1\times \R^d\right)$. 
$U=(u,v,y,z,\omega)\in D\left(\mathcal{A}\right)$. Finally, since $\omega \in L^2(\Gamma_1\times \R^d)$, we get 
$$
-\left(\abs{\xi}^2+\eta\right)\omega(x,\xi)+v|_{\Gamma_1}\mu(\xi)=\omega(x,\xi)-f_5(x,\xi)\in L^2(\Gamma_1\times \R^d). 
$$
Thus, there exists unique $U:=(u,v,y,z,\omega)\in D(\mathcal{A})$ solution of $(I-\mathcal{A})U=F$. The proof is thus complete.    
\end{proof}
$\newline$
\noindent From Proposition \ref{Maximal}, the operator $\mathcal{A}$ is m-dissipative on reflexive Hilbert space $\mathcal{H}$, then following theorem 4.6 in \cite{pazy}, we get $\overline{D(\mathcal{A})}=\mathcal{H}$. Thus,  according  to Lumer-Philips Theorem (see \cite{liu:99} and theorem 4.3 in  \cite{pazy}), the operator $\mathcal{A}$ is the infnitesimal generator of a $C_0-$semigroup of contractions $e^{t\mathcal{A}}$. Then the solution of the evolution problem  \eqref{evolution} admits the following representation 
$$
U(t)=e^{t\mathcal{A}}U_0,\quad t\geq 0,
$$
which leads to the well-posedness of \eqref{evolution}. Hence, we have the following result. 
\begin{theoreme}
{\rm{
For any $U_0\in\mathcal{H}$, problem \eqref{evolution} admits a unique weak solution
$$
U(t)\in C^0(\R^+;\mathcal{H}). 
$$
Moreover, if  $U_0\in D(\mathcal{A}) $, then
$$
U(t)\in C^1(\R^+,\mathcal{H})\cap C^0(\R^+,D(\mathcal{A})).
$$}}
\end{theoreme}
\subsection{Strong Stability in case $a=1$}\label{strongstab}
In this subsection, under the condition of equality of speed propagation wave (i.e. $a=1$), the coupling parameter term $b$ is small enough  and the boundary satisfies the $\textbf{(MGC)}$ boundary condition, we study the strong stability of system \eqref{aug1}-\eqref{aug9} in the sense that its energy converges to zero when $t$ goes to infinity for all initial data in $\mathcal{H}$ . First, we introduce here the notions of stability that we encounter in this work and the definition of the boundary multiplier geometric conditions denoted by $\textbf{(MGC)}$. 
\begin{definition}\label{DefinitionStability}
{\rm{Assume that $\mathcal{A}$ is the generator of a $C_0-$semigroup of contractions $\left(e^{t\mathcal{A}}\right)_{t\geq 0}$ on a Hilbert space $\mathcal{H}$. The $C_0-$semigroup $\left(e^{t\mathcal{A}}\right)_{t\geq 0}$ is said to be 
\begin{enumerate}
\item[$1.$] strongly stable if 
$$
\lim_{t\to +\infty}\|e^{t\AA}x_0\|_{\HH}=0,\quad \forall\ x_0\in \mathcal{H};
$$
\item[$2.$] exponentially (or uniformly) stable if there exist two positive constants $M$ and $\epsilon$ such that 
$$
\|e^{t\AA}x_0\|_{\mathcal{H}}\leq Me^{-\epsilon t}\|x_0\|_{\HH},\quad \forall t>0,\ \forall x_0\in \HH;
$$
\item[$3.$] polynomially stable if there exists two positive constants $C$ and $\alpha$ such that 
$$
\|e^{t\mathcal{A}}x_0\|_{\HH}\leq Ct^{-\ell}\|\AA x_0\|_{\HH},\quad \forall t>0,\quad \forall x_0\in D\left(\AA\right).
$$ 
In that case, one says that conditions of \eqref{evolution} decays at $t^{-\ell}$. The $C_0-$semigroup $\left(e^{t\AA}\right)_{t\geq 0}$ is said to be polynomially stable with optimal decay rate $t^{-\ell}$ (with $\ell>0$) if it is polynomially stable with decay rate $t^{-\ell}$ and, for any $\varepsilon>0$ small enough, there exists solutions of \eqref{evolution} which do not decay at a rate $t^{-(\ell-\varepsilon)}$
\end{enumerate}}}
\end{definition}
\noindent In fact, since the resolvent of $\mathcal{A}$ is not compact, then the classical methods such as Lasalle's invariance principle \cite{Slemrod:89} or the spectrum decomposition theory of Benchimol \cite{benchimol:78} are not applicable. In this case, we use  a general criteria of Arendt-Batty \cite{arendt:88}. We will rely on the following result obtained by Arendt and Batty.
\begin{theoreme}\label{arendtbatty}{\rm{(\textbf{Arendt and Batty in} \cite{arendt:88}) Assume that $\AA$ is the generator of a $C_0-$semigroup of contractions $\left(e^{t\AA}\right)_{t\geq 0}$ on a Hilbert space $H$. If 
\begin{enumerate}
\item[$1.$] $\AA$ has no pure imaginary eigenvalues, 
\item[$2.$] $\sigma(\AA)\cap i\R$ is countable,
\end{enumerate}
where $\sigma(\AA)$ denotes the spectrum of $\AA$, then the $C_0-$semigroup $\left(e^{t\AA}\right)_{t\geq 0}$ is strongly stable. }} 
\end{theoreme}
\begin{definition}{\rm{$\textbf{(MGC)}$\label{MGC} We say that the boundary $\Gamma$ satisfies the boundary multiplier geometric control condition $\textbf{(MGC)}$ if there exists $x_0\in \R^d$ and a positive constant $m_0>0$ such that 
$$
m\cdot \nu\leq 0\ \ \text{on}\ \ \Gamma_0\quad \text{and}\quad m\cdot \nu\geq m_0\ \ \text{on}\ \ \Gamma_1
$$
with $m(x)=x-x_0$, for all $x\in \R^d$.}}
\end{definition}
\begin{figure}[h]
\begin{center}
\includegraphics[height=5cm,width=9.5cm]{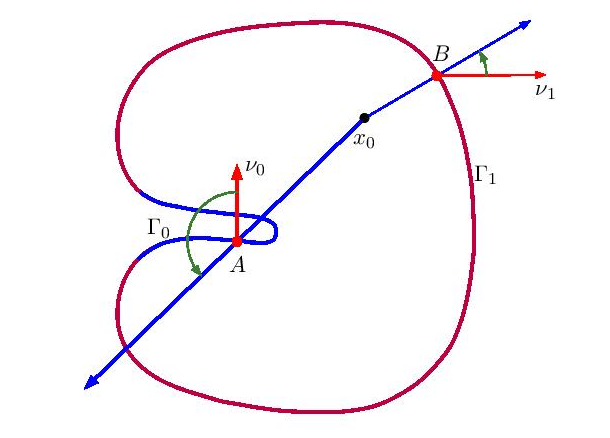}~\includegraphics[height=5cm,width=7cm]{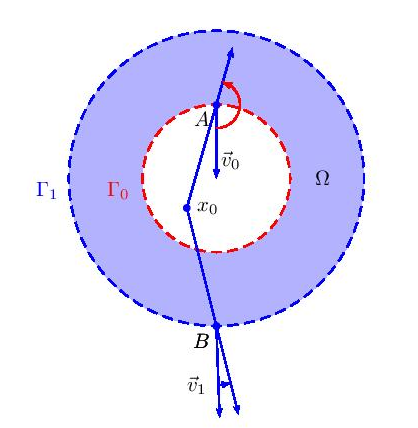}
\caption{Models satisfie ${\rm{\textbf{(MGC)}}}$ condition.}
\end{center}
\end{figure}
\noindent Our main result in this part is the following theorem. 
\begin{theoreme}\label{Strong Stability}
{\rm{Suppose that $\eta\geq 0$, $a=1$, $b$ small enough and $\Gamma$ satisfies \textbf{(MGC)}, then  the $C_0-$semigroup $(e^{t\mathcal{A}})_{t\geq0}$ is strongly stable on the energy space $\mathcal{H}$ in the sense that $\displaystyle{\lim_{t\rightarrow +\infty}\|e^{t\mathcal{A}}U_0\|_{\mathcal{H}}=0}$ for all $U_0 \in \mathcal{H}$.}}
\end{theoreme}
\noindent The argument of  Theorem \ref{Strong Stability} relies on the subsequent lemmas.
\begin{lemma}\label{ker} {\rm{Assume that $\eta\geq 0$, $a=1$, $b$ small enough and $\Gamma$ satisfies \textbf{(MGC)}. Then, for all $\lambda\in \R$, we have 
$$\ker(i\lambda I-\mathcal{A})=\{0\}$$ .}} 
\end{lemma}
\begin{proof}
Let $U\in D(\mathcal{A})$ and let $\lambda\in \R$, such that 
\begin{equation}\label{ST2}
		\mathcal{A}U=i\la U.
		\end{equation}
		Equivalently, we have
		\begin{eqnarray}
		v&=&i\la u,\label{st4}\\
		\Delta u-bz&=&i\la v,\label{st5}\\
		z&=&i\la y,\label{st6}\\
		\Delta y+bv&=&i\la z,\label{st7}\\
		-\left(|\xi|^2+\eta\right)\omega+v|_{\Gamma_1}\mu(\xi)&=&i\la\omega.\label{st8}
		\end{eqnarray}
		Next, a straightforward computation gives  
		\begin{equation}\label{ST3}
		\Re\left(\left<\mathcal{A}U,U\right>_{\mathcal{H}}\right)=-\gamma\kappa(\alpha,d)\int_{\Gamma_1}\int_{\R^d}\left(|\xi|^2+\eta\right)|\omega|^2d\xi d\Gamma.
		\end{equation}
		Then using equation \eqref{ST2} and \eqref{ST3} we deduce that 
		\begin{equation}\label{chap3eq9}
		\omega=0\quad \text{a.e.}\ \text{in}\ \ \Gamma_1\times\R^d.
		\end{equation} 
		It follows, from \eqref{domaine} and \eqref{st8}, that
		\begin{equation}\label{ST4}
		\dfrac{\partial u}{\partial\nu}=0\quad\text{and}\quad v=0\quad\text{on}\quad \Gamma_1.
		\end{equation}
		If $\la =0$, then $v=z=0$, then we obtain the following systems 
		\begin{equation}\label{ST9}
		\left\{\begin{array}{rllll}
		\Delta u&=&0&\text{in}&\Omega\times \R^{+},\\
		u&=&0&\text{on}&\Gamma_0\times \R^+,\\
		\dfrac{\partial u}{\partial\nu}&=&0&\text{on}&\Gamma_1\times\R^+,
		\end{array}
		\right.
		\end{equation}
		and 
		\begin{equation}\label{ST10}
		\left\{\begin{array}{rllll}
		\Delta y&=&0&\text{in}&\Omega\times \R^+,\\
		y&=&0&\text{on}&\Gamma\times \R^+.
		\end{array}\right.
		\end{equation}
		It clear that problems \eqref{ST9} and \eqref{ST10} have a unique solution $u=0$ and $y=0$ respectively, then $U=0$, hence we get $$\ker(-\mathcal{A})=\{0\}.$$
	Otherwise,	if $\la \neq0$, then using equation \eqref{ST4} and \eqref{st4}, we get 
		\begin{equation}\label{ST11}
		u=0\quad\text{on}\quad\Gamma_1.
		\end{equation}
		Eliminating $v$ and $z$ in equations \eqref{st4} and \eqref{st6} in equations \eqref{st5} and \eqref{st7}, we obtain the following system
		\begin{eqnarray}
		\la^2u+\Delta u-i\la by&=&0\ \ \ \text{in}\quad \Omega,\label{ST12}\\
		\la^2y+\Delta y+i\la bu&=&0\ \ \ \text{in}\quad \Omega,\label{ST13}\\
		u&=&0\ \ \ \text{on}\quad \Gamma\label{ST14}\\
		\frac{\partial u}{\partial\nu}&=&0\ \ \text{on}\quad \Gamma_1,\label{ST15}\\
		y&=&0\ \ \ \text{on}\quad \Gamma.\label{ST16}
		\end{eqnarray}
		We divide the proof into several steps.\\
		\textbf{Step 1}. Multiplying equations \eqref{ST12} and \eqref{ST13} by $\bar{y}$ and $\bar{u}$ respectively, integrate over $\Omega$, then using Green formula and the boundary  conditions, we get
		\begin{eqnarray}
		\la^2\int_{\Omega}u\bar{y}dx-\int_{\Omega}\nabla u\cdot\nabla\bar{y}dx-i\la b\int_{\Omega}|y|^2dx&=&0,\label{ST17}\\
		\la^2\int_{\Omega}y\bar{u}dx-\int_{\Omega}\nabla y\cdot\nabla\bar{u}dx+ i\la b\int_{\Omega}|u|^2dx&=&0.\label{ST18}
		\end{eqnarray}
		Adding equations \eqref{ST17} and \eqref{ST18}, then taking the imaginary part, we obtain 
		\begin{equation}\label{ST19}
		\int_{\Omega}|u|^2dx=\int_{\Omega}|y|^2dx.
		\end{equation}
		\textbf{Step 2.} Multiplying equation \eqref{ST12} by $(1-d)\,\bar{u}$, integrate over $\Omega$,  then using green formula and the boundary conditions, we get 
		\begin{equation}\label{ST20}
		\la^2(1-d)\int_{\Omega}|u|^2dx-(1-d)\int_{\Omega}|\nabla u|^2dx=\la b(1-d)\, \Im\left(\int_{\Omega}y\bar{u}dx\right),
		\end{equation}
		where $\Im$ stands for the imaginary part of a complex number.\\
		\textbf{Step 3.} Multiplying equation \eqref{ST12} by $2(m\cdot \nabla\bar{u})$, integrate over $\Omega$, we get 
		\begin{equation}\label{ST21}
		2\la^2\int_{\Omega}u(m\cdot\nabla\bar{u})dx+2\int_{\Omega}\Delta u(m\cdot \nabla\bar{u})dx=2i\la b\int_{\Omega}y(m\cdot \nabla\bar{u})dx.
		\end{equation}
	Since	$U\in D(\mathcal{A})$, then the regularity is sufficiently for applying an integration on the second integral in the left hand said in equation \eqref{ST21}. Then, we have 
		\begin{equation}\label{ST22}
		2\int_{\Omega}\Delta u(m\cdot \nabla\bar{u})dx=-2\int_{\Omega}\nabla u\cdot\nabla\left(m\cdot\nabla\bar{u}\right)dx+2\int_{\Gamma}\frac{\partial u}{\partial\nu}(m\cdot\nabla\bar{u})d\Gamma.
		\end{equation}
		Using the green formula, we get  
		\begin{equation}\label{ST23}
		-2\Re\left(\int_{\Omega}\nabla u\cdot\nabla(m\cdot\nabla\bar{u})dx\right)=(d-2)\int_{\Omega}|\nabla u|^2dx-\int_{\Gamma}(m\cdot\nu)|\nabla u|^2dx.
		\end{equation}
		Inserting equation \eqref{ST23} in equation \eqref{ST22}, then using equations \eqref{ST14} and \eqref{ST15}, we get 
		\begin{equation}\label{ST24}
		2\Re\left(\int_{\Omega}\Delta u(m\cdot\nabla\bar{u})dx\right)=(d-2)\int_{\Omega}|\nabla u|^2dx+\int_{\Gamma_0}(m\cdot\nu)\left|\frac{\partial u}{\partial\nu}\right|^2d\Gamma.
		\end{equation}
		On the other hand, it easy to see that 
		\begin{equation}\label{ST25}
		2\la^2\int_{\Omega}u(m\cdot\nabla\bar{u})dx=-d\la^2\int_{\Omega}|u|^2dx.
		\end{equation}
		Inserting equations \eqref{ST24} and \eqref{ST25} in equation \eqref{ST21}, we get 
		\begin{equation}\label{ST26}
		d\la^2\int_{\Omega}|u|^2dx+(2-d)\int_{\Omega}|\nabla u|^2dx-\int_{\Gamma_0}(m\cdot\nu)\left|\frac{\partial u}{\partial\nu}\right|^2d\Gamma=2\la b\Im\left(\int_{\Omega}y(m\cdot \nabla\bar{u})dx\right).
		\end{equation}
		Adding Equations \eqref{ST20} and  \eqref{ST26}, we get
		\begin{equation}\label{ST27}
		\la^2\int_{\Omega}|u|^2dx+\int_{\Omega}|\nabla u|^2dx-\int_{\Gamma_0}(m\cdot\nu)\left|\frac{\partial u}{\partial\nu}\right|^2d\Gamma=\la b\Im\left(\int_{\Omega}y((d-1)\bar{u}+2(m\cdot\nabla\bar{u}))dx\right).
		\end{equation}
		Using Cauchy-Shwartz inequality in the right hand side equation \eqref{ST27}, then for $\varepsilon>0$, we get 
		\begin{equation}\label{ST28}
		\begin{split}
		\displaystyle{\la^2\int_{\Omega}|u|^2dx+\int_{\Omega}|\nabla u|^2dx-\int_{\Gamma_0}(m\cdot\nu)\left|\frac{\partial u}{\partial\nu}\right|^2d\Gamma}\leq
		\displaystyle{(d-1)|\la||b|\left(\int_{\Omega}\frac{|y|^2}{2\varepsilon}dx\right)^{\frac{1}{2}}\left(\int_{\Omega}2\varepsilon|u|^2dx\right)^{\frac{1}{2}}}\\
		+2\|m\|_{\infty}|\la||b|\left(\int_{\Omega}\frac{|y|^2}{\varepsilon}dx\right)^{\frac{1}{2}}\left(\int_{\Omega}\varepsilon|\nabla u|^2dx\right)^{\frac{1}{2}}.
		\end{split}
		\end{equation}
		Using Young inequality and  Poincar\'e inequality, we get 
		\begin{equation}\label{ST29}
		\begin{split}
		\la^2\int_{\Omega}|u|^2dx+\int_{\Omega}|\nabla u|^2dx-\int_{\Gamma_0}(m\cdot\nu)\left|\frac{\partial u}{\partial\nu}\right|^2\leq \left(\frac{(d-1)^2|\la|^2|b|^2}{4\varepsilon}+\frac{\|m\|_{\infty}|\la|^2|b|^2}{\varepsilon}\right)\int_{\Omega}|y|^2dx\\
		+\varepsilon(1+C)\int_{\Omega}|\nabla u|^2dx
		\end{split}
		\end{equation}
		where $C=\frac{1}{\alpha}$ and $\alpha$ is the smallest eigen value of $-\Delta$ in $H_0^1(\Omega)$.
		Now, using equation \eqref{ST19} and the boundary geometric condition defined in Definition \ref{MGC}, we get 
		\begin{equation}\label{ST30}
		\la^2\int_{\Omega}|u|^2dx+\int_{\Omega}|\nabla u|^2dx\leq \la^2b^2\left(\frac{(N-1)^2}{4\varepsilon}+\frac{\|m\|_{\infty}^2}{\varepsilon}\right)\int_{\Omega}|y|^2dx+\varepsilon(1+C)\int_{\Omega}|\nabla u|^2dx.
		\end{equation}
		Tacking $\varepsilon=\frac{1}{1+C}$ in equation \eqref{ST30}, then using equation \eqref{ST19}, we get 
		\begin{equation}\label{ST31}
		\int_{\Omega}|u|^2dx\leq b^2(1+C)\left(\frac{(d-1)^2}{4}+\|m\|_{\infty}^2\right)\int_{\Omega}|u|^2dx.
		\end{equation}
		Using the fact $b$ is small enough, we obtain 
		\begin{equation}\label{ST32}
		\int_{\Omega}|u|^2dx=0,
		\end{equation}
		such that  $b$ satisfy 
		\begin{equation}\label{ST33}
		1-b^2(1+C)\left(\frac{(d-1)^2}{4}+\|m\|_{\infty}^2\right)>0.
		\end{equation}
		Finally, combining equations \eqref{ST19} and \eqref{ST32}, we deduce that
		\begin{equation}\label{ST34}
		u=y=0,
		\end{equation}
		using equation \eqref{st4}, \eqref{st6} and equation \eqref{ST34}, we obtain 
		\begin{equation}\label{ST35}
		v=z=0.
		\end{equation}
		Consequently, using equation \eqref{ST3}, \eqref{ST34} and \eqref{ST35}, we obtain $U=0$, therefore $$\ker(i\lambda I-\mathcal{A})=\{0\},\quad \forall \lambda\neq0.$$  The proof has been completed.
\end{proof}
\begin{lemma}\label{surjec}
Assume that $\eta=0$. Then, the operator $-\mathcal{A}$ is not invertible and consequently $0\in \sigma(\mathcal{A})$.	
\end{lemma}
\begin{proof}
	First, let $\varphi_{k}\in H_{\Gamma_0}^1(\Omega)$  be an eigenfunction  of the following problem
	\begin{equation}\label{EVPS}
	\left\{\begin{array}{lllll}
	-\Delta\varphi_k&=&\mu_k^2\varphi_k&\text{in}&\Omega,\\
	\varphi_k&=&0,&\text{on}&\Gamma_0,\\
	\frac{\partial \varphi_k}{\partial\nu}&=&0&\text{on}&\Gamma_1.
	\end{array}
	\right.
	\end{equation}
	Next, define the vector $F=\left(\varphi_k,0,0,0,0\right)\in \mathcal{H}$. Assume that there exists $U=(u,v,y,z,\omega)\in D(\mathcal{A})$ such that 
	$$
	-\mathcal{A}U=F.
	$$
	It follows that 
	\begin{equation}\label{uyeq}
	\left\{\begin{array}{lll}
	v=-\varphi_k,&\text{in}&\Omega\\
	z=0&\text{in}&\Omega,\\
	-|\xi|^2\omega+\mu(\xi)v=0&\text{on}&\Gamma_1,
	\end{array}
	\right.
	\end{equation}
	and 
	\begin{equation}\label{uyevp1}
	\left\{\begin{array}{lllll}
	\Delta u&=&0,&\text{in}&\Omega,\\
	\Delta y&=&0,&\text{in}&\Omega,\\
	u&=&0,&\text{on}&\Gamma_0,\\
	\displaystyle
	\frac{\partial u}{\partial\nu}+\gamma\kappa(\alpha,d)\int_{\R^d}\mu(\xi)\omega(\cdot,\xi)d\xi&=&0,&\text{on}&\Gamma_1,\\
	y&=&0,&\text{on}&\Gamma.
	\end{array}
	\right.
	\end{equation}
	From \eqref{uyeq}, we deduce that $\omega(x,\xi)=|\xi|^{\frac{2\alpha-d-4}{2}}\varphi_k|_{\Gamma_1}$. We easy can check that, for $\alpha\in ]0,1[$, the function $\omega(x,\xi)\notin L^2(\Gamma_1\times \R^d)$. So, the assumption of the existence of $U$ is false and consequently the operator $-\mathcal{A}$ is not invertible.
	\end{proof}

\begin{lemma}\label{surjec2}
Assume that $a=1$, $b$ small enough, $\Gamma$ satisfies boundary geometric control condition defined in Definition \ref{MGC}. If   ($\eta>0$ and $\la\in \R$) or ($\eta=0$  and $\lambda\in \R^{\ast}$),  then $i\la I-\mathcal{A}$ is surjective. 
\end{lemma}
\begin{proof}
Let $F=(f_1,f_2,f_3,f_4,f_5)\in \mathcal{H}$, we look for $U=(u,v,y,z,\omega)\in D(\mathcal{A})$ solution of 
\begin{equation}\label{chap3eq11}
(i\la -\mathcal{A})U=F.
\end{equation}  
Equivalently, we have 
$$
\left\{\begin{array}{lllll}
i\la u-v&=&f_1,&\text{in}&\Omega,\\
i\la v-\Delta u+bz&=&f_2,&\text{in}&\Omega,\\
i\la y-\Delta z&=&f_3,&\text{in}&\Omega,\\
i\la z-\Delta y-bv&=&f_4,&\text{in}&\Omega,\\
i\la \omega+(|\xi|^2+\eta)\omega-v\mu(\xi)&=&f_5,&\text{on}&\Gamma_1,
\end{array}
\right.
$$
as before, by eliminating $v,z$ and $\omega$ from the above system and using the fact that 
$$
\partial_{\nu}u+\gamma k(\alpha,d)\int_{\R^d}\mu(\xi)\omega(x,\xi)d\xi=0\quad \text{on}\quad \Gamma_1,
$$
we get the following system:
\begin{equation}\label{im9}
\left\{\begin{array}{lllll}
\vspace{0.15cm}-\la^2u-\Delta u+i\la by&=&f_2+i\la f_1+bf_3&\text{in}&\Omega,\\
\vspace{0.15cm}-\la^2y-\Delta y-i\la bu&=&f_4+i\la f_3-bf_1&\text{in}&\Omega,\\
\vspace{0.15cm}u&=&0&\text{on}&\Gamma_0,\\
\vspace{0.15cm}y&=&0&\text{on}&\Gamma,\\
\vspace{0.15cm}\displaystyle
\frac{\partial u}{\partial \nu}+\mathtt{I}_1(\la,\eta,\alpha)u&=&\mathtt{I}_2(\la,\eta,\alpha)f_1-\mathtt{I}_3(\la,\eta,\alpha)&\text{on}&\Gamma_1,
\end{array}\right.
\end{equation}
where
\begin{equation}\label{I1I2I3}
\left\{\begin{array}{l}
\displaystyle{\mathtt{I}_1(\lambda,\eta,\alpha)=i\lambda \gamma\kappa(\alpha,d)\int_{\R^d}\frac{\mu^2(\xi)}{i\la +\abs{\xi}^2+\eta}d\xi,
}\\[0.1in]
\displaystyle{\mathtt{I}_2(\la,\eta,\alpha)=\gamma k(\alpha,d)\int_{\R^d}\frac{\mu^2(\xi)}{i\la+\abs{\xi}^2+\eta}d\xi},\\[0.1in]
\displaystyle{\mathtt{I}_3(\lambda,\eta,\alpha)=\gamma k(\alpha,d)\int_{\R^d}\frac{\mu(\xi)f_5(x,\xi)}{i\la+\abs{\xi}^2+\eta}d\xi.
}
\end{array}
\right.
\end{equation}
Since $\alpha\in ]0,1[$ and $f_5\in L^2(\Gamma_1\times \R^d)$ in both cases if ($\eta>0$ and $\la\in \R$) or ($\eta=0$ and $\la\in \R^{\ast}$), it is easy to check that 
$$
\abs{\mathtt{I}_1(\lambda,\eta,\alpha)}<\infty,\quad \abs{\mathtt{I}_2(\lambda,\eta,\alpha)}<\infty,\quad \Re\left(\mathtt{I}_1(\lambda,\eta,\alpha)\right)>0\quad \text{and}\quad \int_{\Gamma_1}\mathtt{I}_3(\la,\eta,\alpha)d\Gamma<\infty. 
$$
\textbf{Case 1.} If $\eta>0$ and $\la=0$, then Problem \eqref{im9} becomes 
\begin{equation}\label{newlauyhg}
\left\{\begin{array}{lllll}
-\Delta u&=&f_2+bf_3,&\text{in}&\Omega,\\[0.1in]
-\Delta y&=&f_4-bf_1,&\text{in}&\Omega,\\[0.1in]
u&=&0,&\text{on}&\Gamma_0,\\[0.1in]
\displaystyle{\frac{\partial u}{\partial \nu}}&=&\mathtt{I}_2(0,\eta,\alpha)f_1-I_3(0,\eta,\alpha),&\text{on}&\Gamma_1,\\[0.1in]
y&=&0,&\text{on}&\Gamma.
\end{array}
\right.
\end{equation}
Let $\left(\varphi,\psi\right)\in H_{\Gamma_0}^1(\Omega)\times H_0^1(\Omega)$. Multiplying the first and the second equations of \eqref{newlauyhg} by $\bar{\varphi}$ and $\bar{\psi}$ respectively, integrating in $\Omega$ and taking the sum, then using by parts integration , we get 
\begin{equation}\label{eq1case1}
\int_{\Omega}\nabla u\cdot \nabla\bar{\varphi}dx+\int_{\Omega}\nabla y\cdot\nabla\bar{\psi}dx=\int_{\Omega}(f_2+bf_3)\bar{\phi}dx+\int_{\Omega}(f_4-bf_1)\bar{\psi}dx+\int_{\Gamma}\left(\mathtt{I_2}(0,\eta,\alpha)f_1-\mathtt{I}_3(0,\eta,\alpha)\right)\bar{\varphi}d\Gamma
\end{equation}
The left hand side of \eqref{eq1case1} is a bilinear continuous coercive form on $\left(H_{\Gamma_0}^1(\Omega)\times H_0^1(\Omega)\right)^2$, and the right hand side of equation \eqref{eq1case1} is a linear continuous form on $H_{\Gamma_0}^1(\Omega)\times H_0^1(\Omega)$. Using Lax-Milgram theorem, we deduce that there exists a unique solution $(u,y)\in H_{\Gamma_0}^1(\Omega)\times H_0^1(\Omega)$. Thus, defining $v=-f_1$, $z=-f_3$ and 
$$
\omega(x,\xi):=\frac{f_5(x,\xi)}{\abs{\xi}^2+\eta}-\frac{f_1\mu(\xi)}{\abs{\xi}^2+\eta},
$$
and using the classical regularity arguments, we conclude that equation \eqref{chap3eq11} admits a unique solution $U\in D(\mathcal{A})$.
\\
\textbf{Case 2.} If $\eta\geq 0$ and $\la\in \R^{\ast}$. Let $(\varphi,\psi)\in H_{\Gamma_0}^1(\Omega)\times H_0^1(\Omega)$. Multiplying the first and the second equation of  system \eqref{im9} by $\bar{\varphi}$ and $\bar{\psi}$ respectively and integrate over $\Omega$, we get   
\begin{equation}\label{NewNew1}
\begin{split}
-\la^2\int_{\Omega}u\bar{\varphi}dx+\int_{\Omega}\nabla u\cdot \nabla\bar{\varphi}dx+\mathtt{I}_1(\la,\eta,\alpha)\int_{\Gamma_1}u\bar{\varphi}d\Gamma+i\la b\int_{\Omega}y\bar{\varphi}dx=\int_{\Omega}(i\la f_1+f_2+bf_3)\bar{\varphi}dx\\
+\mathtt{I}_2(\la,\eta,\alpha)\int_{\Gamma_1}f_1\bar{\varphi}d\Gamma-\int_{\Gamma_1}\mathtt{I}_3(\la,\eta,\alpha)\bar{\varphi}d\Gamma
\end{split}
\end{equation}
and 
\begin{equation}\label{NewNew2}
-\la^2\int_{\Omega}y\bar{\psi}dx+\int_{\Omega}\nabla y\cdot\nabla\bar{\psi}dx-i\la b\int_{\Omega}u\bar{\psi}dx=\int_{\Omega}\left(-bf_1+i\la f_3+f_4\right)\bar{\psi}dx.
\end{equation}
Taking the sum of equations \eqref{NewNew1} and \eqref{NewNew2}, we obtain 
\begin{equation}\label{NewNew3}
a((u,y),(\varphi,\psi))=L(\varphi,\psi),\qquad \forall (\varphi,\psi)\in H_{\Gamma_0}^1(\Omega)\times H_0^1(\Omega), 
\end{equation}
where 
$$
a((u,y),(\varphi,\psi))=a_1((u,y),(\varphi,\psi))+a_2((u,y),(\varphi,\psi))
$$
with 
$$
\left\{\begin{array}{l}
\displaystyle{a_1((u,y),(\varphi,\psi))=\int_{\Omega}\left(\nabla u\cdot \nabla\bar{\varphi}+\nabla y\cdot\nabla\bar{\psi}\right)dx+\mathtt{I}_1(\la,\eta,\alpha)\int_{\Gamma_1}u\bar{\varphi}d\Gamma},\\
\displaystyle{a_2((u,y),(\varphi,\psi))=-\la^2\int_{\Omega}\left(u\bar{\varphi}+y\bar{\psi}\right)dx+i\la b\int_{\Omega}\left(y\bar{\varphi}-u\bar{\psi}\right)dx}
\end{array}
\right.
$$
and 
$$
L(\varphi,\psi)=\int_{\Omega}(i\la f_1+f_2+bf_3)\bar{\varphi}dx+\int_{\Omega}\left(-bf_1+i\la f_3+f_4\right)\bar{\psi}dx+\mathtt{I}_2(\la,\eta,\alpha)\int_{\Gamma_1}f_1\bar{\varphi}d\Gamma-\int_{\Gamma_1}\mathtt{I}_3(\la,\eta,\alpha)\bar{\varphi}d\Gamma.
$$
Let $V=H_{\Gamma_0}^1(\Omega)\times H_0^1(\Omega)$ and $V'=H^{-1}(\Omega)\times H^{-1}(\Omega)$ the dual space of $V$.  Let us consider the following operators,
$$
\left\{\begin{array}{llll}
{\rm A}:&V&\rightarrow& V'\\
&(u,y)&\rightarrow &{\rm A}(u,y)
\end{array}
\right.
\left\{\begin{array}{llll}
{\rm A_1}:&V&\rightarrow& V'\\
&(u,y)&\rightarrow &{\rm A_1}(u,y)
\end{array}
\right.
\left\{\begin{array}{llll}
{\rm A_2}:&V&\rightarrow& V'\\
&(u,y)&\rightarrow &{\rm A_2}(u,y)
\end{array}
\right.
$$
such that
\begin{equation}\label{NewNew4}
\left\{\begin{array}{ll}
\displaystyle{\left({\rm A}(u,y)\right)(\varphi,\psi)=a\left((u,y),(\varphi,\psi)\right)},&\forall (\varphi,\psi)\in V,\\[0.1in]
\displaystyle{\left({\rm A_1}(u,y)\right)(\varphi,\psi)=a_1\left((u,y),(\varphi,\psi)\right)},&\forall (\varphi,\psi)\in V,\\[0.1in]
\displaystyle{\left({\rm A_2}(u,y)\right)(\varphi,\psi)=a_2\left((u,y),(\varphi,\psi)\right)},&\forall (\varphi,\psi)\in V.
\end{array}
\right.
\end{equation}
Our goal is to prove that ${\rm A}$ is an isomorphism operator. For this aim, we divid the proof into three steps.\\
\textbf{Step 1.} In this step, we proof that the operator $A_1$ is an isomorphism operator. For this goal, following the second equation of \eqref{NewNew4} and by using the fact that $\Re(\mathtt{I}_1(\la,\eta,\alpha))>0$ and the trace theorem we can easily verify that $a_1$ is a bilinear continuous coercive form on $V$. Then, by Lax-Milligram Lemma, the operator ${\rm A_1}$ is an isomorphism.\\
\textbf{Step 2.} In this step, we proof that the operator ${\rm A_2}$ is compact. According to the third equation of \eqref{NewNew4}, we have 
$$
\abs{a_2((u,y),(\varphi,\psi))}\leq C\|(u,y)\|_{L^2(\Omega)}\|(\varphi,\psi)\|_{L^2(\Omega)}.
$$   
Finally, using the compactness embedding from $V$ to $L^2(\Omega)\times L^2(\Omega)$ and the continuous embedding from $L^2(\Omega)\times L^2(\Omega)$ into $V'$, we deduce that ${\rm A_2}$ is compact.\\ \\
From steps 1 and 2, we get that the operator ${\rm A}={\rm A_1}+{\rm A_2}$ is a Fredholm operator of index zero. Consequently, by Fredholm alternative, to prove that operator ${\rm A}$ is an isomorphism it is enough to prove that ${\rm A}$ is injective, i.e. $\ker\{{\rm A}\}=\{0\}$. \\
\textbf{Step 3.} In this step, we prove that $\ker\{{\rm A}\}=\{0\}$. For this aim, let $(\tilde{u},\tilde{y})\in \ker\left({\rm A}\right)$, i.e. 
$$
a\left((\tilde{u},\tilde{y}),(\varphi,\psi)\right)=0,\qquad \forall (\varphi,\psi)\in V.
$$
Equivalently, we have 
\begin{equation}\label{NewNew5}
-\la^2\int_{\Omega}\left(u\bar{\varphi}+y\bar{\psi}\right)dx+i\la b\int_{\Omega}\left(y\bar{\varphi}-u\bar{\psi}\right)dx+\int_{\Omega}\left(\nabla u\cdot \nabla\bar{\varphi}+\nabla y\cdot\nabla\bar{\psi}\right)dx+\mathtt{I}_1(\la,\eta,\alpha)\int_{\Gamma_1}u\bar{\varphi}d\Gamma=0.
\end{equation}
Taking $\varphi=\tilde{u}$ and $\psi=\tilde{y}$ in equation \eqref{NewNew5}, we get 
$$
-\la^2\int_{\Omega}\abs{u}^2dx-\la^2\int_{\Omega}\abs{y}^2dx-2\la b\Im\left(\int_{\Omega}y\bar{u}dx\right)+\int_{\Omega}\abs{\nabla u}^2dx+\int_{\Omega}\abs{\nabla y}^2dx+\mathtt{I_1}(\la,\eta,\alpha)\int_{\Gamma_1}\abs{u}^2dx=0.
$$
By taking the imaginary part of the above equation, we get 
\begin{equation}\label{NewNew6}
0=\Im\left(\mathtt{I}_1(\la,\eta,\alpha)\right)\int_{\Gamma_1}\abs{\tilde{u}}^2d\Gamma=\la\gamma k(\alpha,d)\left(\int_{\R^d}\frac{\mu^2(\xi)\left(\abs{\xi}^2+\eta\right)}{\la^2+\left(\abs{\xi}^2+\eta\right)^2}d\xi\right)\left(\int_{\Gamma_1}\abs{\tilde{u}}^2d\Gamma\right),
\end{equation}
 Since $\la\gamma k(\alpha,d)\left(\int_{\R^d}\frac{\mu^2(\xi)\left(\abs{\xi}^2+\eta\right)}{\la^2+\left(\abs{\xi}^2+\eta\right)^2}d\xi\right)\neq 0$, we get 
\begin{equation}\label{NewNew7}
\tilde{u}=0,\qquad \text{a.e.}\quad \text{on}\quad \Gamma_1.
\end{equation}
Then, we find that 
\begin{equation}\label{NewNew8}
\left\{\begin{array}{lllll}
\vspace{0.15cm}-\la^2\tilde{u}-\Delta \tilde{u}+i\la b\tilde{y}&=&0&\text{in}&\Omega,\\
\vspace{0.15cm}-\la^2\tilde{y}-\Delta \tilde{y}-i\la b\tilde{u}&=&0&\text{in}&\Omega,\\
\vspace{0.15cm}\tilde{u} &=&0&\text{on}&\Gamma_0,\\
\vspace{0.15cm}\displaystyle
\frac{\partial \tilde{u}}{\partial \nu}&=&0&\text{on}&\Gamma_1,\\
\vspace{0.15cm}\tilde{y} &=&0&\text{on}&\Gamma.
\end{array}
\right.
\end{equation}
Therefore, the vector $\tilde{U}$ define by 
$$
\tilde{U}=(\tilde{u},i\la \tilde{u},\tilde{y},i\la \tilde{y},0)
$$
belongs to $D(\mathcal{A})$ and we have 
$$
i\la \tilde{U}-\mathcal{A}\tilde{U}=0.
$$
Hence, $\tilde{U}\in \ker(i\la I-\mathcal{A})$, then by Lemma \ref{ker}, under the assumptions $b$ is small enough, we get $\tilde{U}=0$ and $\Gamma$ satisfies the boundary geometric condition \textbf{(MGC)}, this implies that $\tilde{u}=\tilde{y}=0$. Consequently, $\ker({\rm A})=\{0\}.$\\
Therefore, from step 3 and Fredholm alternative, we get that operator ${\rm A}$ is an isomorphism. By using trace theorem, it is easy to see that the operator $L$ is continuous from $V$ to $L^2(\Omega)\times L^2(\Omega)$. Consequently, equation \eqref{NewNew3} admits a unique solution $(u,y)\in V$. Thus, defining $v=i\la u-f_1$, $z=i\la y-f_3$ and 
$$
\omega(x,\xi):=\frac{f_5(x,\xi)}{i\la +\abs{\xi}^2+\eta}+\frac{i\la u\mu(\xi)}{i\la +\abs{\xi}^2+\eta}-\frac{f_1\mu(\xi)}{i\la+\abs{\xi}^2+\eta},
$$
and using the classical regularity arguments, we conclude that equation \eqref{chap3eq11} admits a unique solution $U\in D(\mathcal{A})$. The proof is thus complete. 
\end{proof}

\noindent \textbf{Proof of Theorem \ref{Strong Stability}.}  Using Lemma \ref{ker} we have that $\mathcal{A}$ has non pure imaginary eigenvalues. According to Lemmas \ref{ker}, \ref{surjec}, \ref{surjec2} and with the help of the closed graph theorem of Banach, we deduce that $\sigma(\mathcal{A})\cap i\mathbb{R}=\emptyset$ if $\eta>0$ and $\sigma(\mathcal{A}_1)\cap i\R=\{0\}$ if $\eta=0$. Thus, we get the conclusion by applying Theorem 
\ref{arendtbatty} of Arendt Batty. The proof of the theorem is complete.

\begin{rem}\label{rem1d}  In the case that $d=1$ and $\eta\geq 0$, in \cite{akilwehbe02}, the authors showed that the semigroup of contractions $e^{t\mathcal{A}}$ is strongly stable on the energy space $\mathcal{H}$ if and only if the coupling parameter term satisfies the following condition 
\begin{equation}\tag{${\rm SC}$}
b^2\neq \frac{(k_1^2-ak_2^2)(k_2^2-ak_1^2)}{(a+1)(k_1^2+k_2^2)}\pi^2,\quad \forall k_1,k_2\in \mathbb{Z}.
\end{equation}
\end{rem}
\begin{rem}\label{rem1dnew}  In the case that $d=1$ and $\eta\geq 0$, in \cite{akilwehbe02}, by using spectrum methods, the authors showed that the semigroup generated by the operator $\mathcal{A}$ is not exponentially stable in the energy space $\mathcal{H}$.
\end{rem}
\begin{pro} \label{Theorem-3.2}(See \cite{akilwehbe02})
Assume that  condition ${\rm (SC)}$ holds. Then there exists $n_0\in \mathbb{N}$ sufficiently large  and two sequences $\left(\lambda_{1,n}\right)_{ |n|\geq n_0} $ and $\left(\lambda_{2,n}\right)_{ |n|\geq n_0} $ of simple roots of $f$ (that are also simple eigenvalues of $\mathcal{A}$) satisfying the following asymptotic behavior:\\[0.1in]
\textbf{Case 1.} If $b\neq k\pi$, $k\in\mathbb{Z}^*$, then
\begin{equation}\label{eq-3.15}
		\displaystyle{\lambda_{1,n}= i n\pi+{\frac {\gamma\, \left( 1-\cos \left( b \right)  \right)  \left( i\cos
 \left( \frac{\pi\alpha}{2} \right) -\sin \left( \frac{\pi\alpha}{2}
 \right)  \right) }{ 2\left( n\pi \right) ^{1-\alpha}}}+o \left( 
 \frac{1}{{n}^{1-\alpha}} \right) 
, \ \ \forall \ |n|\geq n_0}
\end{equation}
and
\begin{equation}\label{eq-3.16}
		\displaystyle{\lambda_{2,n}= i n\pi+\frac{i\pi}{2}}+{\frac {\gamma\, \left( 1+\cos \left( b \right)  \right)  \left( i\cos
 \left( \frac{\pi\alpha}{2} \right) -\sin \left( \frac{\pi\alpha}{2}
 \right)  \right) }{ 2\left( n\pi \right) ^{1-\alpha}}}+o \left( 
 \frac{1}{{n}^{1-\alpha}} \right) 
, \ \ \forall\  |n|\geq n_0.
\end{equation}
\textbf{Case 2.} If $b=2 k\pi$, $k\in\mathbb{Z}^*$, then
\begin{equation}\label{eq-3.17}
		\displaystyle{\lambda_{1,n}= i n\pi+{\frac {i{b}^{2}}{8n\pi}}+{\frac {7 i b^4}{128{\pi}^{
3}{n}^{3}}}+{\frac {\gamma\,{b}^{6} \left( i\cos \left(\frac{\pi\alpha}{2}  \right) -\sin \left( \frac{\pi\alpha}{2} \right)  \right) }{128\,{
\pi}^{5-\alpha}{n}^{5-\alpha}}}+O \left( \frac{1}{n^5}\right) 
, \ \ \forall \ |n|\geq n_0}
\end{equation}
and
\begin{equation}\label{eq-3.18}
		\displaystyle{\lambda_{2,n}= i n\pi+\frac{i\pi}{2}}+{\frac {\gamma \left( i\cos
 \left( \frac{\pi\alpha}{2} \right) -\sin \left( \frac{\pi\alpha}{2}
 \right)  \right) }{ \left( n\pi \right) ^{1-\alpha}}}+o \left( 
 \frac{1}{{n}^{1-\alpha}} \right) , \ \ \forall\  |n|\geq n_0.
\end{equation}
\textbf{Case 3.} If $b=(2 k+1)\pi$, $k\in\mathbb{Z}^*$, then
\begin{equation}\label{eq-3.19}
		\displaystyle{\lambda_{1,n}= i n\pi+{\frac {\gamma \left( i\cos
 \left( \frac{\pi\alpha}{2} \right) -\sin \left( \frac{\pi\alpha}{2}
 \right)  \right) }{ \left( n\pi \right) ^{1-\alpha}}}+o \left( 
 \frac{1}{{n}^{1-\alpha}} \right), \ \ \forall \ |n|\geq n_0}
\end{equation}
and
\begin{equation}\label{eq-3.20}
\begin{array}{ll}
		\displaystyle{\lambda_{2,n}= i n\pi+\frac{i\pi}{2}+{\frac {i{b}^{2}}{8n\pi}}-{\frac {i{b}^{2}}{16\pi {n}^{2}}}+{\frac 
{i{b}^{2} \left( 4{\pi}^{2}+7{b}^{2} \right) }{{128\pi
}^{3}{n}^{3}}}}
\\ \\ \hspace{1cm}
\displaystyle{-{\frac {i{b}^{2} \left( 4{\pi}^{2}+21
{b}^{2} \right) }{256{\pi}^{3}{n}^{4}}}+{\frac {\gamma {b}^{6} \left( 
i\cos \left(\frac{\pi\alpha}{2} \right) -\sin \left( \frac{\pi\alpha}{2}
 \right)  \right) }{256 {\pi}^{5-\alpha}{n}^{5-\alpha}}}+O
 \left(\frac{1}{n^5} \right) 
,\ \ \forall\ |n|\geq n_0.}
\end{array}
\end{equation}
\textbf{Case 4.} If $a\neq 1$, then 
\begin{equation}\label{eq-3.21}
\la_{1,n}=in\pi\sqrt{a}+o(1)\quad \text{and}\quad \la_{2,n}=i\left(n+\frac{1}{2}\right)\pi+o(1),\quad \forall \abs{n}\geq n_0.
\end{equation}			
\end{pro}	

\section{Polynomial Stability}\label{PSMCC}
\noindent 
 In this section, we use the frequency domain approach method combining with multiplier method in the case that $\eta>0,\ a=1, b$ small enough and under the $\textbf{(MGC)}$ condition defined in Definition \ref{MGC}, with $\overline{\Gamma_0}\cap \overline{\Gamma}_1=\emptyset$. The frequency domain approach method has been obtained by Batty in \cite{batkai:06}-\cite{batty:08}, Borichev and Tomilov in \cite{borichev:10} and Liu and Rao in \cite{liurao:05}.
\begin{theoreme}\label{BattyBorichevLiu}
Assume that $\mathcal{A}$ is the generator of a strongly continuous semigroup of contractions $\left(e^{t\mathcal{A}}\right)_{t\geq 0}$ on the energy space $\mathcal{H}$. If $i\R\subset \rho(\mathcal{A})$, then for a fixed $\ell>0$ the following conditions are equivalent. 
\begin{enumerate}
\item[$1.$] $\displaystyle{\sup_{\la\in \R}\left\|\left(i\la I-\mathcal{A}\right)^{-1}\right\|_{\mathcal{L}\left(\mathcal{H}\right)}}=O\left(\abs{\la}^{\ell}\right)$,
\item[$2.$] $\left\|e^{t\mathcal{A}}U_0\right\|_{\mathcal{H}}\leq \frac{C}{t^{\frac{1}{\ell}}}\quad \forall t>0,\ U_0\in D(\mathcal{A}),\ \text{for some}\ C>0$.
\end{enumerate}
\end{theoreme}
Our results are gathered in the following theorem.
\begin{theoreme}\label{pol}
Assume that $a=1, \ \eta > 0$ and $b$ small enough. Then, for all initial data $U_0\in D(\mathcal{A})$, there exists a constant $C>0$ independent of $U_0$, such that the energy of the strong solution $U$ of  \eqref{evolution}, satisfies the following estimation 
\begin{equation}\label{sp1}
E(t,U)\leq C\frac{2}{t^{\frac{2}{1-\alpha}}}\|U_0\|^2_{D(\mathcal{A})},\quad \forall t>0.
\end{equation}
\end{theoreme}	
\noindent Since, for $\eta>0$, $\sigma(\mathcal{A})\cap i\R=\emptyset$, then for the proof of Theorem \ref{pol}, according to Theorem \ref{BattyBorichevLiu}, we need to prove that 
\begin{equation}\tag{${\rm H3}$}
\sup_{\la\in \R}\left\|(i\la Id-\mathcal{A})^{-1}\right\|_{\mathcal{L}(\mathcal{H})}=O\left(\abs{\la}^{1-\alpha}\right).
\end{equation}
First of all, we define the function $\theta$ by
\begin{equation}\label{theta} 
\left\{\begin{array}{lll}
\theta\equiv 0&\text{on}&\Gamma_0,\\
\theta\equiv 1&\text{on}&\Gamma_1,\\
\theta(x)\in [0,1].
\end{array}
\right.
\end{equation}
We will argue by contradiction. For this purpose, suppose ${\rm (H3)}$ is false, then there exists a real sequence $(\lambda_n)$, with $|\lambda_n|\rightarrow+\infty$ and a sequence $(U^n)\subset D(\mathcal{A})$, verifying the following conditions
\begin{equation}\label{sp2}
	\|U^n\|_{\mathcal{H}}=\|(u^n,v^n,y^n,z^n,\omega^n)\|_{\mathcal{H}}=1
\end{equation}
and
\begin{equation}\label{SP3}
	\la_n^{\ell}(i\la_n-\mathcal{A})U^n=(f_1^n,f_2^n,f_3^n,f_4^n,f_5^n)\to 0\quad\text{in}\quad\mathcal{H}.
\end{equation}  
For the simplicity, we drop the index $n$. Detailing equation \eqref{SP3}, we get 
\begin{eqnarray}
	i\la u-v&=&\frac{f_1}{\la^{\ell}}\longrightarrow 0\quad\text{in}\quad H_{\Gamma_0}^1(\Omega),\label{SP4}\\
	i\la v-\Delta u+bz&=&\frac{f_2}{\la^{\ell}}\longrightarrow 0\quad\text{in}\quad L^2(\Omega),\label{SP5}\\
	i\la y-z&=&\frac{f_3}{\la^{\ell}}\longrightarrow 0\quad\text{in}\quad H_0^1(\Omega),\label{SP6}\\
	i\la z-\Delta y-bv&=&\frac{f_4}{\la^{\ell}}\longrightarrow 0\quad\text{in}\quad L^2(\Omega),\label{SP7}\\
	i\la\omega+(|\xi|^2+\eta)\omega-v|_{\Gamma_1}\mu(\xi)&=&\frac{f_5}{\la^{\ell}}\longrightarrow 0\quad\text{in}\quad L^2(\Gamma_1\times\R^d).\label{SP8}
\end{eqnarray}
Note that $U$ is uniformly bounded in $\mathcal{H}$. Then, taking the inner product of  \eqref{SP3} with $U$ in $\mathcal{H}$, we get 
\begin{equation}\label{SP9}
	-\gamma\kappa\int_{\Gamma_1}\int_{\R^d}(|\xi|^2+\eta)|\omega|^2d\xi d\Gamma=\Re\left((i\la I-\mathcal{A})U,U\right)_{\mathcal{H}}=\frac{o(1)}{\la^{1-\alpha}}.
\end{equation}
By eliminating $v$ and $z$ from the above system, we get 
\begin{eqnarray}
	\la^2u+\Delta u-i\la by&=&-\frac{f_2}{\la^{\ell}}-\frac{i\la f_1}{\la^{\ell}}-\frac{bf_3}{\la^{\ell}},\label{sp11}\\
	\la^2y+\Delta y+i\la bu&=&-\frac{f_4}{\la^{\ell}}-\frac{i\la f_3}{\la^{\ell}}+\frac{bf_1}{\la^{\ell}}.\label{sp12}
\end{eqnarray}
\begin{lemma}\label{information}Assume that $\eta>0$. Then the solution $(u,v,y,z,\omega)\in D(\mathcal{A})$ of \eqref{SP4}-\eqref{SP8} satisfies the following asymptotic behavior estimation
	\begin{eqnarray}
		\|u\|_{L^2(\Omega)}&=&\frac{O(1)}{\la},\label{estim1}\\
		\|y\|_{L^2(\Omega)}&=&\frac{O(1)}{\la},\label{estim2}\\
		\|\partial_{\nu}u\|_{L^2(\Gamma_1)}&=&\frac{o(1)}{\la^{\frac{1-\alpha}{2}}},\label{estim3}\\
		\|u\|_{L^2(\Gamma_1)}&=&\frac{o(1)}{\la}\label{estim4}.
	\end{eqnarray}
\end{lemma}
\begin{proof}
	Using equations \eqref{SP3}, \eqref{SP4} and \eqref{SP6}, we deduce directly the estimations \eqref{estim1}-\eqref{estim2}. Now, from the boundary condition 	
$$
\partial_{\nu}u+\gamma\kappa\int_{\R^d}\mu(\xi)\omega(x,\xi)d\xi=0\quad \text{on}\quad \Gamma_1,
$$	
using Cauchy-Schwartz  inequality, we get 
\begin{equation}\label{INFOO1}
\|\partial_{\nu}u\|_{L^2(\Gamma_1)}^2\leq\gamma^2\kappa^2\left(\int_{\R^d}\frac{\mu^2(\xi)}{|\xi|^2+\eta}d\xi\right)\left(\int_{\Gamma_1}\int_{\R^d}(|\xi|^2+\eta)|\omega|^2d\xi d\Gamma\right).
\end{equation}	
Then, combining equation \eqref{SP9} and equation \eqref{INFOO1}, we obtain the desired equation \eqref{estim3}. Finally multiplying equation \eqref{SP8} by $(i\la +|\xi|^2+\eta)^{-1-d}\abs{\xi}^d$, integrating over $\R^d$ with respect to the variable $\xi$ and applying Cauchy-Schwartz inequality, we obtain 
\begin{equation}\label{NewINFOO1}
\begin{array}{ccl}
\displaystyle
\abs{v|_{\Gamma_1}}\int_{\R^d}\frac{\abs{\xi}^{\alpha+\frac{d}{2}}}{\left(\abs{\la}+|\xi|^2+\eta\right)^{d+1}}d\xi&\leq&\displaystyle
\left(\int_{\R^d}\frac{\abs{\xi}^{2d-2}}{\left(\abs{\la}+\abs{\xi}^2+\eta\right)^{2d}}d\xi\right)^{\frac{1}{2}}\left(\int_{\R^d}\abs{\xi\omega(x,\xi)}^{2}d\xi\right)^{\frac{1}{2}}\\
&&\displaystyle
+\frac{1}{\abs{\la}^{\ell}}\left(\int_{\R^d}\abs{f_5(x,\xi)}^2d\xi\right)^{\frac{1}{2}}\left(\int_{\R^d}\frac{\abs{\xi}^{2d}}{(\abs{\la}+\abs{\xi}^2+\eta)^{2d+2}}d\xi\right)^{\frac{1}{2}}.
\end{array}
\end{equation}
Using Young's inequality for products in equation \eqref{NewINFOO1} and integrate over $\Gamma_1$, we get
\begin{equation}\label{INFOO2}
\int_{\Gamma_1}|v|^2d\Gamma_1\leq \frac{2A_2}{A_1}\int_{\Gamma_1}\int_{\R^d}\abs{\xi\omega(x,\xi)}^2d\xi d\Gamma+\frac{2A_3}{A_1\abs{\la}^{2\ell}}\int_{\Gamma_1}\int_{\R^d}\abs{f_3(x,\xi)}^2d\xi d\Gamma
\end{equation}
where 
\begin{equation}\label{A1A2A3}
A_1=\left(\int_{\R^d}\frac{\abs{\xi}^{\alpha+\frac{d}{2}}}{\left(\abs{\la}+\abs{\xi}^2+\eta\right)^{d+1}}\right)^2,\quad A_2=\int_{\R^d}\frac{\abs{\xi}^{2d-2}}{\left(\abs{\la}+\abs{\xi}^2+\eta\right)^{2d}d\xi}\ \text{and}\ A_3=\int_{\R^d}\frac{\abs{\xi}^{2d}}{\left(\abs{\la}+\abs{\xi}^2+\eta\right)^{2d+2}}d\xi. 
\end{equation}
Indeed, as $\eta>0$, we have 
$$
\int_{\Gamma_1}\int_{\R^d}\abs{\xi\omega(x,\xi)}^2d\xi d\Gamma \leq \int_{\Gamma_1}\int_{\R^d}\left(\abs{\xi}^2+\eta\right)\abs{\omega(x,\xi)}^2d\xi d\Gamma,
$$
which gives, 
\begin{equation}\label{NEWINFOO2}
\int_{\Gamma_1}|v|^2d\Gamma_1\leq \frac{2A_2}{A_1}\int_{\Gamma_1}\left(\abs{\xi}^2+\eta\right)\abs{\omega(x,\xi)}^2d\xi d\Gamma+\frac{2A_3}{A_1\abs{\la}^{2\ell}}\int_{\Gamma_1}\int_{\R^d}\abs{f_3(x,\xi)}^2d\xi d\Gamma.
\end{equation}
On the other hand, using Lemma \ref{Appendix2} , we get 
\begin{equation}\label{NEWA1A2A3}
A_1=c(\alpha,d)\left(\abs{\la}+\eta\right)^{\alpha-\frac{d}{2}-2},\quad A_2=c_1(d)\left(\abs{\la}+\eta\right)^{-1-\frac{d}{2}}\quad \text{and}\quad A_3=c_2(d)\left(\abs{\la}+\eta\right)^{-\frac{d}{2}-2}.
\end{equation}
Inserting equation \eqref{NEWA1A2A3} in \eqref{NEWINFOO2} and using the fact that $\ell=1-\alpha$, we get 
\begin{equation}\label{NEWINFOO3}
\|v\|_{L^2(\Gamma_1)}=o(1).
\end{equation}
It follows, from \eqref{SP4}, that equation \eqref{estim4} holds. The proof has been completed.

\end{proof}
\begin{lemma}\label{ynu} Assume that $\eta>0$, $b$ small enough and $\Gamma$ satisfies the geometric boundary condition ${\rm (MGC)}$. Then, the solution $(u,v,y,z,\omega)\in D(\mathcal{A})$ of \eqref{SP4}-\eqref{SP8} satisfies the following asymptotic behavior estimation 
	\begin{equation}\label{nn4}
		\|\partial_{\nu}y\|_{L^2(\Gamma_1)}=O(1).		
	\end{equation}
\end{lemma}
\begin{proof}
	Multiplying equation \eqref{sp12} by $2\theta(m\cdot\nabla\bar{y})$, we obtain  
	\begin{equation}\label{SP19}
		\begin{split}
			2\int_{\Omega}\la^2y\theta(m\cdot\nabla\bar{y})dx+2\int_{\Omega}\Delta y\theta(m\cdot\nabla\bar{y})dx+2i\int_{\Omega}\la bu\theta(m\cdot\nabla\bar{y})dx=\\
			-2\int_{\Omega}\theta\left(\frac{f_4}{\la^{\ell}}+\frac{i\la f_3}{\la^{\ell}}-\frac{bf_1}{\la^{\ell}}\right)(m\cdot \nabla\bar{y})dx.
		\end{split}
	\end{equation}
	First, using the facts that $\nabla y$ is bounded in $L^2(\Omega)$, $\|f_1\|_{H_{\Gamma_0}^1(\Omega)}=o(1)$ and $\|f_4\|_{L^2(\Omega)}=o(1)$, we get 
	\begin{equation}\label{estim5}
		-2\int_{\Omega}\theta\left(\frac{f_4}{\la^{\ell}}-b\frac{f_1}{\la^{\ell}}\right)\left(m\cdot \nabla\bar{y}\right)dx=\frac{o(1)}{\la^{\ell}}.
	\end{equation}
	On the other hand, using Green formula and the fact that $y=0$ on $\Gamma$, we get 
	\begin{equation*}
		-2\int_{\Omega}\frac{i\la\theta f_3}{\lambda^{\ell}}(m\cdot \nabla\bar{y})dx=2\int_{\Omega}i\frac{\la \bar{y}\cdot\nabla(\theta mf_3)}{\la^{\ell}},
	\end{equation*}
using the fact that $\|f_3\|_{H_0^1(\Omega)}=o(1)$ and equation \eqref{estim2}   in the above equation, we get	\begin{equation}\label{estim7}
		-2\int_{\Omega}\frac{i\la\theta f_3}{\la^{\ell}}(m\cdot \nabla\bar{y})dx=\frac{o(1)}{\la^{\ell}}.
	\end{equation}
	Next, using Green formula and the fact $y=0$ on $\Gamma$,  we get 
	\begin{equation}\label{SP20}
		2\int_{\Omega}\la^2y\theta\left(m\cdot\nabla\bar{y}\right)dx=-\int_{\Omega}\left(d\theta+(m\cdot\nabla\theta)\right)|\la y|^2dx.
	\end{equation}
	Now, using the Green formula,  we get  
	\begin{equation}\label{SP21}
		2\Re\left(\int_{\Omega}\Delta y\theta(m\cdot\nabla\bar{y})dx\right)=-2\Re\left(\int_{\Omega}\nabla y\cdot\nabla\left(\theta(m\cdot\nabla\bar{y})\right)dx\right)+2\int_{\Gamma}\theta(\partial_{\nu} y)\left(m\cdot \nabla\bar{y}\right)d\Gamma.
	\end{equation}
	Furthermore, using Green formula, we get 
	\begin{equation}\label{SP22}
		\left\{\begin{array}{lll}
			\displaystyle
			-2\Re\left(\int_{\Omega}\nabla y\cdot\nabla\left(\theta(m\cdot\nabla\bar{y})\right)dx\right)&=&\displaystyle
			\int_{\Omega}(m\cdot\nabla \theta)|\nabla y|^2dx+(d-2)\int_{\Omega}\theta|\nabla y|^2dx\\
			&&\displaystyle 
			-\int_{\Gamma}\theta(m\cdot\nu)|\nabla y|^2d\Gamma\\
			&&\displaystyle
			-2\Re\left(\int_{\Omega}(\nabla y\cdot\nabla\theta)(m\cdot \nabla\bar{y})dx\right).
		\end{array}\right.
	\end{equation}
	Then, combining equation \eqref{SP21}, \eqref{SP22} and  using the fact that  $y=\frac{\partial y}{\partial\tau}=0$ on $\Gamma$,
	 we get 
	\begin{equation}\label{SP23}
		\left\{\begin{array}{lll}
			\displaystyle
			2\Re\left(\int_{\Omega}\Delta y\theta(m\cdot\nabla\bar{y})dx\right)&=&\displaystyle
			-2\Re\left(\int_{\Omega}(\nabla y\cdot\nabla\theta)(m\cdot \nabla\bar{y})dx\right)\\
			&&\displaystyle
			-(2-d)\int_{\Omega}\theta |\nabla y|^2dx+\int_{\Gamma_1}\theta(m\cdot\nu)\left|\partial_{\nu}y\right|^2d\Gamma_1\\
			&&\displaystyle
			+\int_{\Omega}(m\cdot\nabla\theta)|\nabla y|^2dx.
		\end{array}\right.
	\end{equation}
	Inserting  equations \eqref{estim5}, \eqref{estim7}, \eqref{SP20} and \eqref{SP23} in equation \eqref{SP19}, we obtain 
	\begin{equation}\label{sp21}
		\left\{\begin{array}{lll}
			\displaystyle
			\int_{\Gamma_1}\theta(m\cdot\nu)\left|\partial_{\nu}y\right|^2d\Gamma&=&\displaystyle
			\int_{\Omega}\left(d\theta+(m\cdot\nabla\theta)\right)|\la y|^2dx-(d-2)\int_{\Omega}\theta|\nabla y|^2dx\\
			&&\displaystyle
			+2\Re\left(\int_{\Omega}(\nabla y\cdot\nabla\theta)(m\cdot \nabla\bar{y})dx\right)-\int_{\Omega}(m\cdot\nabla\theta)|\nabla y|^2dx\\
			\displaystyle
			&&\displaystyle
			-2\Re\left(i\la\int_{\Omega}bu\theta(m\cdot\nabla\bar{y})dx\right)+\frac{o(1)}{\la^{\ell}}.
		\end{array}
		\right.
	\end{equation}
	Finally, using equation \eqref{estim2}, the fact that $\nabla y$ is bounded in $L^2(\Omega)$, $\theta=1$ on $\Gamma_1$ and $\ell=1-\alpha$, we obtain the desired equation \eqref{nn4}. The proof is thus complete.
\end{proof}
\begin{lemma}\label{laulay} Assume that $\eta>0$, $b$ small enough and $\Gamma$ satisfies the geometric condition ${\rm (MGC)}$. Then the solution $(u,v,y,z,\omega)\in D(\mathcal{A})$ of \eqref{SP4}-\eqref{SP8} satisfies the following asymptotic behavior estimation  
	\begin{equation}\label{nn6}
		\displaystyle{\int_{\Omega}|\la u|^2dx-\int_{\Omega}|\la y|^2dx=o(1)}.
	\end{equation}
\end{lemma}
\begin{proof}
	Multiplying equations \eqref{sp11} and \eqref{sp12} by $\la\bar{y}$ and $\la\bar{u}$ respectively, integrate over $\Omega$, using Green formula, we obtain 
	\begin{equation}\label{SP31}
		\int_{\Omega}\la^3u\bar{y}dx-\la\int_{\Omega}\nabla u\nabla\bar{y}dx-ib\int_{\Omega}|\la y|^2dx=-\int_{\Omega}\left(\frac{f_2}{\la^{\ell}}+i\frac{\la f_1}{\la^{\ell}}+\frac{bf_3}{\la^{\ell}}\right)\la \bar{y}dx
	\end{equation}
	and 
	\begin{equation}\label{SP32}
			\int_{\Omega}\la^3y\bar{u}dx-\la\int_{\Omega}\nabla y\nabla\bar{u}dx+\la\int_{\Gamma_1}(\partial_{\nu}y)\bar{u}d\Gamma_1+ib\int_{\Omega}|\la u|^2dx=
			-\int_{\Omega}\left(\frac{f_4}{\la^{\ell}}+\frac{i\la f_3}{\la^{\ell}}-\frac{bf_1}{\la^{\ell}}\right)\la\bar{u}dx.
	\end{equation}
	First, using equations \eqref{estim1}, \eqref{estim2} and the facts that  $\|f_1\|_{H_{\Gamma_0}^1(\Omega)}=o(1)$, $\|f_2\|_{L^2(\Omega)}=o(1)$, $\|f_3\|_{H_0^1(\Omega)}=o(1)$,  $\|f_4\|_{L^2(\Omega)}=o(1)$, we get 
	\begin{eqnarray}
		\int_{\Omega}\left(\frac{f_2}{\la^{\ell}}+\frac{bf_3}{\la^{\ell}}\right)\la \bar{y}dx&=&\frac{o(1)}{\la^{\ell}},\label{estim8}\\
		\int_{\Omega}\left(\frac{f_4}{\la^{\ell}}-\frac{bf_1}{\la^{\ell}}\right)\la \bar{u}dx&=&\frac{o(1)}{\la^{\ell}}.\label{estim9}
	\end{eqnarray}
	Using  Lemma \ref{information} and Equation \eqref{nn4}, we get 
	\begin{equation}\label{SP33}
		\la\int_{\Gamma_1}\frac{\partial y}{\partial\nu}\bar{u}d\Gamma_1=o(1).
	\end{equation}
	Next, multiplying equations \eqref{sp11} and \eqref{sp12} respectively by $\bar{f}_3$ and $\bar{f}_1$ and  integrating in $\Omega$, we get
	\begin{equation}\label{estim10}
		\int_{\Omega}\la^2u\bar{f_3}-\int_{\Omega}\nabla u\nabla\bar{f}_3dx-i\la b\int_{\Omega}y\bar{f}_3dx=-\int_{\Omega}\left(\frac{f_2}{\la^{\ell}}+i\frac{\la f_1}{\la^{\ell}}+\frac{bf_3}{\la^{\ell}}\right)\bar{f}_3dx
	\end{equation}
	and 
	\begin{equation}\label{Estim11}
		\begin{split}
			\int_{\Omega}\la^2y\bar{f}_1dx-\int_{\Omega}\nabla y\nabla\bar{f}_1dx+\int_{\Gamma_1}(\partial_{\nu}y)\bar{f}_1d\Gamma_1+i\la b\int_{\Omega}u\bar{f}_1dx\\
			=-\int_{\Omega}\left(\frac{f_4}{\la^{\ell}}+i\frac{\la f_3}{\la^{\ell}}-\frac{bf_1}{\la^{\ell}}\right)\bar{f}_1dx.
		\end{split}
	\end{equation}
	Using equations \eqref{estim1}, \eqref{estim2} and \eqref{nn4} in  \eqref{estim10}-\eqref{Estim11}, then using  the facts that  $\|f_1\|_{H_{\Gamma_0}^1(\Omega)}=o(1)$, $\|f_2\|_{L^2(\Omega)}=o(1)$, $\|f_3\|_{H_0^1(\Omega)}=o(1)$, $\|f_4\|_{L^2(\Omega)}=o(1)$ and $\nabla u$, $\nabla y$ are bounded in $L^2(\Omega)$, we obtain 
	\begin{eqnarray}
		\int_{\Omega}\la^2u\bar{f}_3dx&=&-i\int_{\Omega}\frac{\la f_1\bar{f}_3}{\la^{\ell}}dx+o(1),\label{estim11}\\
		\int_{\Omega}\la^2y\bar{f}_1dx&=&-i\int_{\Omega}\frac{\la f_3\bar{f}_1}{\la^{\ell}}dx+o(1).\label{estim12}
	\end{eqnarray}
	Now, combining equation \eqref{SP31}, \eqref{estim8} and \eqref{estim12} we get 
	\begin{equation}\label{estim13}
		\int_{\Omega}\la^3u\bar{y}dx-\la \int_{\Omega}\nabla u\nabla\bar{y}dx-ib\int_{\Omega}|\la y|^2dx=\int_{\Omega}\frac{\la \bar{f}_3f_1}{\la^{2\ell}}dx+\frac{o(1)}{\la^{\ell}}.
	\end{equation}
	On the other hand, combining equation \eqref{SP32}, \eqref{estim9}, \eqref{SP33} and \eqref{estim11}, we obtain
	\begin{equation}\label{estim14}
		\int_{\Omega}\la^3y\bar{u}dx-\la\int_{\Omega}\nabla y\nabla\bar{u}dx+ib\int_{\Omega}|\la u|^2dx=\int_{\Omega}\frac{\la\bar{f}_1f_3}{\la^{2\ell}}dx+o(1).
	\end{equation}
	Finally, adding equations \eqref{estim13} and \eqref{estim14}, using the fact that $\ell=1-\alpha$ and taking the imaginary part, we get the desired equation \eqref{nn6}. The proof is thus complete.
\end{proof}
\begin{lemma}\label{laugradu}Assume that $\eta>0$. Then the solution $(u,v,y,z,\omega)\in D(\mathcal{A})$ of \eqref{SP4}-\eqref{SP8} satisfies the following asymptotic behavior estimation
	\begin{equation}\label{nn7}
		\displaystyle{\int_{\Omega}|\la u|^2dx-\int_{\Omega}|\nabla u|^2dx=\frac{O(1)}{\la}+\frac{o(1)}{\la^{1-\alpha}}}\quad \text{and}\quad  \int_{\Omega}|\la y|^2dx-\int_{\Omega}|\nabla y|^2dx=\frac{O(1)}{\la}+\frac{o(1)}{\la^{1-\alpha}}.
	\end{equation}
\end{lemma}
\begin{proof}
Multiplying equations \eqref{sp11} and \eqref{sp12} by $\bar{u}$ and $\bar{y}$ respectively, integrating in $\Omega$, then using Green formula, we get 
	\begin{equation}\label{SP36}
		\int_{\Omega}|\la u|^2dx-\int_{\Omega}|\nabla u|^2dx+\int_{\Gamma_1}\frac{\partial u}{\partial\nu}\bar{u}d\Gamma_1-ib\int_{\Omega}\la y\bar{u}dx=-\int_{\Omega}\left(\frac{f_2}{\la^{\ell}}+\frac{i\la f_1}{\la^{\ell}}+\frac{bf_3}{\la^{\ell}}\right)\bar{u}dx
	\end{equation}
	and
	\begin{equation}\label{gtg36}
		\int_{\Omega}|\la y|^2dx-\int_{\Omega}|\nabla y|^2dx+ib\int_{\Omega}\la u\bar{y}dx=-\int_{\Omega}\left(\frac{f_4}{\la^{\ell}}+\frac{i\la f_3}{\la^{\ell}}-\frac{bf_1}{\la^{\ell}}\right)\bar{y}dx.
	\end{equation}
Using the fact that $\|f_1\|_{H_{\Gamma_0}^1(\Omega)}=o(1)$, $\|f_2\|_{L^2(\Omega)}=o(1)$, $\|f_3\|_{H_0^1(\Omega)}=o(1)$, $\|f_4\|_{L^2(\Omega)}=o(1)$, Lemma \ref{information} and $\ell=1-\alpha$, we get the desired equation \eqref{nn7}. The proof is thus complete.
\end{proof}
\begin{lemma}\label{lauo} Assume that $\eta>0$. If $|b|\leq \frac{1}{\|m\|_{\infty}}$, then the solution $(u,v,y,z,\omega)\in D(\mathcal{A})$ of \eqref{SP4}-\eqref{SP8} satisfies the following asymptotic behavior estimation
	\begin{equation}\label{nn8}
		\int_{\Omega}|\la u|^2dx=o(1).
	\end{equation}
\end{lemma}
\begin{proof}
	Multiplying equation \eqref{sp11} by $2(m\cdot\nabla\bar{u})$, then integrating in $\Omega$, we get  
	\begin{equation}\label{SP38}
		\begin{split}
			2\int_{\Omega}\la^2u(m\cdot\nabla\bar{u})dx+2\int_{\Omega}\Delta u(m\cdot\nabla\bar{u})dx-2i\la b\int_{\Omega}b y(m\cdot\nabla\bar{u})dx=\\
			-2\int_{\Omega}\left(\frac{f_2}{\la^{\ell}}+\frac{i\la f_1}{\la^{\ell}}+\frac{bf_3}{\la^{\ell}}\right)(m\cdot \nabla\bar{u})dx.
		\end{split}
	\end{equation}
	First, using the fact that $\|f_2\|_{L^2(\Omega)}=o(1)$, $\|f_3\|_{H_0^1(\Omega)}=o(1)$, $\nabla \bar{u}$ is bounded in $L^2(\Omega)$, and the fact that $\ell=1-\alpha$, we get 
	\begin{equation}\label{estim15}
		-2\int_{\Omega}\left(\frac{f_2}{\la^{\ell}}+\frac{bf_3}{\la^{\ell}}\right)(m\cdot\nabla\bar{u})dx=\frac{o(1)}{\la^{1-\alpha}}.
	\end{equation}
	On the other hand
	\begin{equation}\label{estim16}
		2i\la\int_{\Omega}\frac{f_1}{\la^{\ell}}(m\cdot\nabla\bar{u})dx=-2\frac{i\la}{\la^{\ell}}\int_{\Omega}\bar{u}\cdot\nabla(f_1m)dx+\frac{2i\la}{\la^{\ell}}\int_{\Gamma_1}(mf_1\cdot\nu)\bar{u}d\Gamma_1.
	\end{equation}
	 Using equation \eqref{estim1}, \eqref{estim4} and the fact that $\|f_1\|_{H_{\Gamma_0}^1(\Omega)}=o(1)$,  $\ell=1-\alpha$ in \eqref{estim16}, we obtain
	\begin{equation}\label{estim17}
		2i\la\int_{\Omega}\frac{f_1}{\la^{\ell}}(m\cdot \nabla\bar{u})dx=\frac{o(1)}{\la^{1-\alpha}}.
	\end{equation}
	Next, using integration by parts, we get 
	\begin{equation}\label{SP39}
		2\int_{\Omega}\la^2u(m\cdot\nabla\bar{u})dx=-d\int_{\Omega}|\la u|^2dx+\la^2\int_{\Gamma_1}(m\cdot\nu)|u|^2d\Gamma_1.
	\end{equation}
	Thus, using equation \eqref{estim4},  we obtain 
	\begin{equation}\label{SP40}
		\la^2\int_{\Gamma_1}(m\cdot\nu)|u|^2d\Gamma_1=o(1).
	\end{equation}
	Combining equations \eqref{SP39} and \eqref{SP40}, we get 
	\begin{equation}\label{SP41}
		2\int_{\Omega}\la^2u(m\cdot\nabla u)dx=-d\int_{\Omega}|\la u|^2dx+o(1).
	\end{equation}
	Now, using Green formula,  we get 
	\begin{equation}\label{SP42}
		2\int_{\Omega}\Delta u(m\cdot\nabla\bar{u})dx=(d-2)\int_{\Omega}|\nabla u|^2dx+2\int_{\Gamma}\frac{\partial u}{\partial\nu}(m\cdot\nabla\bar{u})d\Gamma-\int_{\Gamma}(m\cdot\nu)|\nabla u|^2d\Gamma.
	\end{equation}
	Therefore, inserting equations \eqref{estim5}, \eqref{estim7}, \eqref{SP41}, \eqref{SP42} in equation \eqref{SP38}, then taking the reel part, we get 
	\begin{equation}\label{SP43}
		\begin{split}
			\displaystyle
			-d\int_{\Omega}|\la u|^2dx+(d-2)\int_{\Omega}|\nabla u|^2dx+2\Re\left(\int_{\Gamma}\frac{\partial u}{\partial\nu}(m\cdot\nabla\bar{u})d\Gamma\right)-\int_{\Gamma}(m\cdot\nu)|\nabla u|^2d\Gamma\\
			\displaystyle
			+2\la b\Re\left(-i\int_{\Omega} y(m\cdot\nabla\bar{u})dx\right)=o(1).
		\end{split}
	\end{equation}
	Using the fact $\frac{\partial u}{\partial\tau}=0$ on $\Gamma_0$, we get
	\begin{equation}\label{SP44}
		\begin{split}
			-2\Re\left(\int_{\Gamma}\frac{\partial u}{\partial\nu}(m\cdot\nabla\bar{u})d\Gamma\right)+\int_{\Gamma}(m\cdot\nu)|\nabla u|^2d\Gamma=-\int_{\Gamma_0}(m\cdot\nu)\left|\frac{\partial u}{\partial\nu}\right|^2d\Gamma\\
			-2\Re\left(\int_{\Gamma_1}\frac{\partial u}{\partial\nu}(m\cdot\nabla\bar{u})d\Gamma\right)+\int_{\Gamma_1}(m\cdot\nu)|\nabla u|^2d\Gamma
		\end{split}
	\end{equation}
	Let $\varepsilon>0$, so by Young inequality, we get 
	\begin{equation}\label{SP45}
		2\Re\left(\int_{\Gamma_1}\frac{\partial u}{\partial\nu}(m\cdot\nabla\bar{u})d\Gamma\right)\leq \frac{\|m\|_{\infty}^2}{\varepsilon}\int_{\Gamma_1}\left|\frac{\partial u}{\partial\nu}\right|^2d\Gamma_1+\varepsilon\int_{\Gamma_1}|\nabla u|^2d\Gamma.
	\end{equation}
	Using equation \eqref{estim3} in \eqref{SP45}, we get 
	\begin{equation}\label{SP47}
		2\Re\left(\int_{\Gamma_1}\frac{\partial u}{\partial\nu}(m\cdot\nabla\bar{u})d\Gamma\right)\leq \varepsilon\int_{\Gamma_1}|\nabla u|^2d\Gamma+\frac{o(1)}{\la^{1-\alpha}}.
	\end{equation}
	Inserting equation \eqref{SP47} in equation \eqref{SP44}, we get 
	\begin{equation}\label{SP48}
		\begin{split}
			-2\Re\left(\int_{\Gamma}\frac{\partial u}{\partial\nu}(m\cdot\nabla\bar{u})d\Gamma\right)+\int_{\Gamma}(m\cdot\nu)|\nabla u|^2d\Gamma\geq
			-\int_{\Gamma_0}(m\cdot\nu)\left|\frac{\partial u}{\partial\nu}\right|^2d\Gamma\\
			-\varepsilon\int_{\Gamma_1}|\nabla u|^2d\Gamma+\int_{\Gamma_1}(m\cdot\nu)|\nabla u|^2d\Gamma+\frac{o(1)}{\la^{1-\alpha}}.
		\end{split}
	\end{equation}
	Using the $(\textbf{MGC})$ condition in equation \eqref{SP48}, we get 
	\begin{equation}\label{SP49}
		-2\Re\left(\int_{\Gamma}\frac{\partial u}{\partial\nu}(m\cdot\nabla\bar{u})d\Gamma\right)+\int_{\Gamma}(m\cdot\nu)|\nabla u|^2d\Gamma\geq \frac{o(1)}{\la^{1-\alpha}}+(\delta-\varepsilon)\int_{\Gamma_1}|\nabla u|^2d\Gamma.
	\end{equation}
	Taking $\varepsilon<\delta$, then we get 
	\begin{equation}\label{SP50}
		-2\Re\left(\int_{\Gamma}\frac{\partial u}{\partial\nu}(m\cdot\nabla\bar{u})d\Gamma\right)+\int_{\Gamma}(m\cdot\nu)|\nabla u|^2d\Gamma\geq\frac{o(1)}{\la^{1-\alpha}}.
	\end{equation}
	Inserting equation \eqref{SP50} in equation \eqref{SP43}, we get 
	\begin{equation}\label{SP51}
		d\int_{\Omega}|\la u|^2dx+(2-d)\int_{\Omega}|\nabla u|^2dx\leq 2\la b\Re\left(-i\int_{\Omega} y(m\cdot\nabla\bar{u})dx\right)+o(1).
	\end{equation}
	Multiplying equation \eqref{nn7} by $1-d$ and combining with equation \eqref{SP51}, we get
	\begin{equation}\label{SP52}
		\int_{\Omega}|\la u|^2dx+\int_{\Omega}|\nabla u|^2dx\leq 2\Re\left(-i\int_{\Omega}b\la y(m\cdot\nabla\bar{u})dx\right)+o(1).
	\end{equation}
	Using  Young inequality, we get 
	\begin{equation}\label{SP53}
		2\Re\left(-i\int_{\Omega}b\la y(m\cdot\nabla\bar{u})dx\right)\leq \|m\|_{\infty}^2b^2\int_{\Omega}|\la y|^2dx+\int_{\Omega}|\nabla u|^2dx.
	\end{equation}
	Inserting equation \eqref{SP53} in equation \eqref{SP52}, we get 
	\begin{equation}\label{SP54}
		\int_{\Omega}|\la u|^2dx+\int_{\Omega}|\nabla u|^2dx\leq\|m\|_{\infty}^2b^2\int_{\Omega}|\la y|^2dx+\int_{\Omega}|\nabla u|^2dx+o(1).
	\end{equation}
	Using Lemma \ref{laulay}, we get 
	\begin{equation}\label{SP55}
		\int_{\Omega}(1-\|m\|_{\infty}^2b^2)|\la u|^2dx\leq o(1).
	\end{equation}
	Finally, using the fact that $|b|<\frac{1}{\|m\|_{\infty}}$ in equation \eqref{SP55}, we get the desired equation \eqref{nn8}.
	Thus the  proof is complete.
\end{proof}

\noindent\textbf{Proof of Theorem \ref{pol}}. Using \eqref{nn6}, \eqref{nn7} and \eqref{nn8}, we get 
\begin{equation}\label{estim18}
	\int_{\Omega}|\nabla u|^2dx=o(1),\ \ \int_{\Omega}|\la y|^2dx=o(1),\ \ \text{and}\ \  \int_{\Omega}|\nabla y|^2dx=o(1).
\end{equation}
It follows from \eqref{SP9}, \eqref{nn8} and \eqref{estim18}, that $\|U\|_{\mathcal{H}}=o(1)$ which contradicts \eqref{sp2}. Consequently condition ${\rm (H3)}$ holds and the energy of smooth solution of system \eqref{aug1}-\eqref{aug9} decays polynomial to zero as $t$ goes to infinity where $a=1$ and $b$ small enough. In view of the asymptotic behaviour of the eigenvalues of the operator $\mathcal{A}$ (see proposition \ref{Theorem-3.2} ), we  conjecture  that the  energy decay rate of type $t^{-\frac{2}{1-\alpha}}$ is optimal  (see \cite{Loreti-Rao:06}). The proof has been completed.
\section*{Conclusion}
\noindent We  studied the stabilization of multidimensional coupled wave equations via velocities with one locally boundary fractional damping acting on a part of the boundary of the domain. First, under the ${\rm\textbf{(MGC)}}$ boundary condition without, the strong stability is proved if $a=1$ and the coupling parameter term is small enough. Next, under the ${\rm(MGC)}$ boundary condition with $\overline{\Gamma_0}\cap \overline{\Gamma_1}=\emptyset$ a  polynomial energy decay rate of type $\frac{1}{t^{\frac{2}{1-\alpha}}}$ is established if $a=1$ and the coupling parameter term is small enough. Moreover, it would be interesting to 
\begin{enumerate}
\item study the strong stability of system \eqref{chap3eq1}-\eqref{chap3eq7} in case that  $a\neq 1$, $b$ is not necessarily small enough and without any additional geometric boundary control condition,
\item generalize the results obtained in \cite{akilwehbe02} (see system \eqref{AGW}) to the multidimensional case. 
\end{enumerate}

\section{Appendix}
\noindent Let $\mu$ be the function defined by 
\begin{equation}\label{muappendix}
\mu(\xi)=|\xi|^{\frac{2\alpha-d}{2}},\quad \xi\in \R^{d}\quad 
		\text{and}\quad 0<\alpha<1.
		\end{equation} 
\begin{lemma}\label{Appendix1}
Let $\eta \geq  0$, then we have 
\begin{equation*}
M_2(\alpha,d)=\gamma k(\alpha,d)\int_{\R^d}\frac{\mu^2(\xi)}{1+\eta+\abs{\xi}^2}d\xi=\gamma (1+\eta)^{\alpha-1},
\end{equation*}
where $k(\alpha,d)$ is defined in equation \eqref{InverseCaputoDerivative}.
\end{lemma}	
\begin{proof}
Using the hyper-spherical coordinates and the fact that the Jacobian $J$ is defined by 
\begin{equation}\label{jacobian}
J=\rho^{d-1}\prod_{j=1}^{d-2}\sin^{d-1-j}\left(\phi_j\right),
\end{equation}
we get 
\begin{equation*}
M_2(\alpha,d)=\gamma k(\alpha,d)\int_0^{+\infty}\frac{\rho^{2\alpha-1}}{1+\eta+\rho^2}\left(\prod_{j=1}^{d-2}\left(\int_0^{\pi}\sin^{d-1-j}(\phi_j)d\phi_j\right)\int_0^{2\pi}d\phi_{d-1}\right)d\rho.
\end{equation*}
On the other hand, it is easy to see that 
\begin{equation}\label{ProdJacobian}
\prod_{j=1}^{d-2}\left(\int_0^{\pi}\sin^{d-1-j}(\phi_j)d\phi_j\right)\int_0^{2\pi}d\phi_{d-1}=\frac{d\pi^{\frac{d}{2}}}{\Gamma\left(\frac{d}{2}+1\right)}.
\end{equation}
This implies that 
$$
M_2(\alpha,d)=2\frac{\gamma \sin(\alpha\pi)}{\pi}\int_0^{+\infty}\frac{\rho^{2\alpha-1}}{1+\eta+\rho^2}d\rho.
$$
A direct computation gives 
$$
\begin{array}{lllll}
M_2(\alpha,d)&=&\displaystyle{\frac{\gamma \sin(\alpha\pi)}{\pi}\int_0^{+\infty}\frac{x^{\alpha-1}}{1+\eta+x}}\\\\
    \vspace{0.15cm}&=&\displaystyle{\frac{\gamma \sin(\alpha\pi)}{\pi}(1+\eta)^{\alpha-1}\int_0^1\frac{(1-z)^{\alpha-1}}{z^{\alpha}}dz}\\\\
    \vspace{0.15cm}&=&\displaystyle{\frac{\gamma \sin(\alpha\pi)}{\pi}\left(1+\eta\right)^{\alpha-1}\Gamma(1-\alpha)}
    \Gamma(\alpha)\\\\
    \vspace{0.15cm}&=&\displaystyle{\gamma\left(1+\eta\right)^{\alpha-1}}.
    \end{array}
	$$
\end{proof}
\begin{lemma}\label{Appendix11}
Let $\eta\geq 0$, then we have 
$$
S_1=\int_{\R^d}\frac{\abs{\xi}^{2\alpha-d+2}}{(1+\abs{\xi}^2+\eta)^2}d\xi<+\infty. 
$$
\end{lemma}
\begin{proof}
Using the same argument in Lemma \ref{Appendix1}, we get 
$$
S_1=\frac{d\pi^{\frac{d}{2}}}{\Gamma\left(\frac{d}{2}+1\right)}\int_0^{+\infty}\frac{\rho^{2\alpha+1}}{(1+\eta+\rho^2)^2}d\rho. 
$$
Since, we have 
\begin{equation*}
\frac{\rho^{2\alpha+1}}{(1+\rho^2+\eta)^2}\isEquivTo{0}\frac{\rho^{2\alpha+1}}{(1+\eta)^2}\quad \text{and}\quad \frac{\rho^{2\alpha+1}}{(1+\rho^2+\eta)^2}\isEquivTo{+\infty}\frac{1}{\rho^{3-2\alpha}}.
\end{equation*}
Using the fact that $\alpha\in (0,1)$, we get $S_1$ is well defined. 
\end{proof}

\begin{lemma}\label{Appendix2}
Let $\eta >0$ and $d\geq 2$, then we have 
\begin{equation*}
\left\{
\begin{array}{ll}

\displaystyle{
B_1:=\displaystyle{\int_{\R^d}\frac{\abs{\xi}^{\alpha+\frac{d}{2}}}{\left(\abs{\la}+|\xi|^2+\eta\right)^{d+1}}d\xi}=c_1(\alpha,d)\left(\abs{\la}+\eta\right)^{\frac{\alpha}{2}-\frac{d}{4}-1},
}

\nline
\displaystyle{
A_2:=\int_{\R^d}\frac{\abs{\xi}^{2d-2}}{\left(\abs{\la}+\abs{\xi}^2+\eta\right)^{2d}d\xi}=c_2(d)\left(\abs{\la}+\eta\right)^{-1-\frac{d}{2}},
}
\nline
\displaystyle{
A_3:=\int_{\R^d}\frac{\abs{\xi}^{2d}}{\left(\abs{\la}+\abs{\xi}^2+\eta\right)^{2d+2}}d\xi=c_3(d)\left(\abs{\la}+\eta\right)^{-\frac{d}{2}-2}
}.
\end{array}
\right.
\end{equation*}
such that
\begin{equation}\label{ctildec}
c_1(\alpha,d)=\frac{d\pi^{\frac{d}{2}}\Gamma\left(\frac{d}{4}-\frac{\alpha}{2}+1\right)\Gamma\left(\frac{\alpha}{2}+\frac{3d}{4}\right)}{2\Gamma\left(\frac{d}{2}+1\right)\Gamma\left(d+1\right)},\ \ 
c_2(d)=\frac{d\pi^{\frac{d}{2}}\Gamma\left(\frac{d}{2}\right)\Gamma\left(\frac{3d}{2}\right)}{2\Gamma\left(\frac{d}{2}+1\right)\Gamma\left(2d\right)},\quad 
c_3(d)=\frac{d\pi^{\frac{d}{2}}\Gamma\left(2+\frac{d}{2}\right)\Gamma\left(\frac{3d}{2}\right)}{2\Gamma\left(\frac{d}{2}+1\right)\Gamma\left(2+2d\right)}.
\end{equation}
\end{lemma}
\begin{proof}
Our first aim is to calculate $B_1$. Using the hyper-spherical coordinates and the fact that the Jacobian $J$  defined in equation \eqref{jacobian}, we get 
\begin{equation}\label{calculA1}
B_1=\int_0^{+\infty}\frac{\rho^{\alpha+3\frac{d}{2}-1}}{\left(|\la|+\eta+\rho^2\right)^{d+1}}\left(\prod_{j=1}^{d-2}\left(\int_0^{\pi}\sin^{d-1-j}(\phi_j)d\phi_j\right)\int_0^{2\pi}d\phi_{d-1}\right)d\rho.
\end{equation}
Using equation \eqref{ProdJacobian} in equation \eqref{calculA1}, we obtain 
$$
B_1=\frac{d\pi^{\frac{d}{2}}}{2\Gamma\left(\frac{d}{2}+1\right)}\int_0^{+\infty}\frac{\rho^{\alpha+3\frac{d}{2}-1}}{\left(|\la|+\eta+\rho^2\right)^{d+1}}d\rho.
$$
A direct computation gives 
$$
\begin{array}{ccl}
B_1&=&\displaystyle
\frac{d\pi^{\frac{d}{2}}}{2\Gamma\left(\frac{d}{2}+1\right)}\int_0^{+\infty}\frac{x^{\frac{\alpha}{2}+3\frac{d}{4}-1}}{\left(\abs{\la}+\eta+x\right)^{d+1}}dx\\ [0.1in]
&=&\displaystyle
\frac{d\pi^{\frac{d}{2}}}{2\Gamma\left(\frac{d}{2}+1\right)}\left(\abs{\la}+\eta\right)^{\frac{\alpha}{2}-\frac{d}{4}-1}\int_{1}^{+\infty}\frac{(y-1)^{\frac{\alpha}{2}+3\frac{d}{4}-1}}{y^{d+1}}dy\\ [0.1in]
&=&\displaystyle
\frac{d\pi^{\frac{d}{2}}}{2\Gamma\left(\frac{d}{2}+1\right)}\left(\abs{\la}+\eta\right)^{\frac{\alpha}{2}-\frac{d}{4}-1}\int_0^1(1-z)^{\frac{\alpha}{2}+3\frac{d}{4}-1}z^{\frac{d}{4}-\frac{\alpha}{2}}dz\\ [0.1in]
&=&\displaystyle
\frac{d\pi^{\frac{d}{2}}\Gamma\left(\frac{d}{4}-\frac{\alpha}{2}+1\right)\Gamma\left(\frac{\alpha}{2}+3\frac{d}{4}\right)}{2\Gamma\left(\frac{d}{2}+1\right)\Gamma\left(d+1\right)}\left(\abs{\la}+\eta\right)^{\frac{\alpha}{2}-\frac{d}{4}-1}
\end{array}
$$
By the same way, we can calculate  $A_2$ and $A_3$.

\end{proof}

\end{document}